\documentclass[12pt]{amsart}
\usepackage{amssymb,amsmath,amsthm,mathrsfs,multirow,xcolor,framed,url}
\usepackage[pdftex,
         pdfauthor={Dandan Chen},
         pdftitle={On Decomposition of $\theta_2^{2n}(\tau)$ as the Sum of Lambert Series and Cusp forms},
         pdfsubject={MATHEMATICS},
         pdfkeywords={Elliptic function; Theta function; Eisenstein series.},
         pdfproducer={Latex with hyperref},
         pdfcreator={pdflatex}]{hyperref}
\newtheorem{theorem}{Theorem}[section]
\newtheorem{lemma}[theorem]{Lemma}
\newtheorem{cor}[theorem]{Corollary}

\theoremstyle{definition}

\theoremstyle{remark}

\numberwithin{equation}{section}

\newcommand\nutwid{\overset {\text{\lower 3pt\hbox{$\sim$}}}\nu}

\allowdisplaybreaks
\newcommand\mylabel[1]{\label{#1}}

\newcommand\thm[1]{\ref{thm:#1}}

\newcommand\lem[1]{\ref{lem:#1}}

\newcommand\corol[1]{\ref{cor:#1}}

\newcommand\eqn[1]{(\ref{eq:#1})}

\newcommand\sect[1]{\ref{sec:#1}}

\newcommand{\beqs}{\begin{equation*}}
\newcommand{\eeqs}{\end{equation*}}
\newcommand{\beq}{\begin{equation}}
\newcommand{\eeq}{\end{equation}}

\begin{document}
\title [On Decomposition of $\theta_2^{2n}(\tau)$ as the Sum of Eisenstein series and Cusp forms]
        {On Decomposition of $\theta_2^{2n}(\tau)$ as the Sum of Eisenstein series and Cusp forms}

\author{Dandan Chen}
\address{School of Mathematical Sciences, East China Normal University,
Shanghai, People's Republic of China}
\email{ddchen@stu.ecnu.edu.cn}
\author{Rong Chen}
\address{School of Mathematical Sciences, East China Normal University,
Shanghai, People's Republic of China}
\email{rchen@stu.ecnu.edu.cn}

\subjclass[2010]{33E05 14H42 11M36}

\date{\today}


\keywords{Elliptic function; Theta function; Eisenstein series.}
\begin{abstract}
Based on the values of the Weierstrass elliptic function $\wp(z|\tau)$ at $z=\pi\tau/2$, $(\pi+\pi\tau)/{2}, (\pi+\pi\tau)/{4},(\pi+2\pi\tau)/{4}$ and the theory of modular forms on the arithmetic group $\Gamma_0(2)$, we decompose $\theta_2^{2n}(\tau)$ as sum of Eisenstein series and  cusp forms.  Using the recurrence relation of $\wp^{(2n)}(z|\tau)$, we provide an algorithm to determine the exact form of these cusp forms.
\end{abstract}
\maketitle

\section{Introduction}
\label{sec:intr}

We will adopt the definitions of theta functions given in \cite[Chapter 21]{whittaker and watson}. The reader's familiarity with the basic properties of Weierstrass elliptic function and the theta functions is assumed. Let $\theta_j(z|\tau), j=1,2,3,4$, denote the Jacobi's theta functions and, for brevity, $\theta_j(\tau)=\theta_j(0|\tau)$.
It is easy to find that
\begin{align}\label{theta2}
\theta_2(\tau)=2q^{1/4}\sum_{n=0}^{\infty}q^{n(n+1)}
=2q^{1/4}(q^2;q^2)_\infty (-q^2;q^2)^2_\infty.
\end{align}
Here and later we use the standard $q$-series notation and $q=\exp(\pi  i \tau)$  with  $\tau\in $ $\mathbb H=\{\tau~|~\tau=x+{ i}y ~{and}~ y>0\}$, where $\mathbb H$ denote the upper half plane:
\begin{align*}
(a;q)_0=1, \quad (a;q)_n=\prod_{k=0}^{n-1}(1-aq^k), \quad (a;q)_\infty=\prod_{k=0}^\infty(1-aq^k),
\end{align*}
 for all positive integers $n$.

In \cite{sun} Sun made an intersting connection  between the Wallis' formula and $\theta_2^2(\tau)$. His observation is as follows.
By \eqref{theta2}, we have
\begin{align*}
\prod_{n=1}^{\infty}\frac{(1-q^{4n})^2}{(1-q^{4n-2})(1-q^{4n+2})}
=(1-q^2)\prod_{n=1}^{\infty}\frac{(1-q^{4n})^2}{(1-q^{4n-2})^2}
=\frac{1}{4}(1-q^2)q^{-1/2}\theta^2_2(\tau).
\end{align*}
From the well-known  Wallis' formula
\begin{align*}
\prod_{n=1}^{\infty}\frac{4n^2}{4n^2-1}=\frac \pi2
\end{align*}
and observe that
\begin{align*}
\lim_{q\rightarrow 1}\prod_{n=1}^{\infty}\frac{(1-q^{4n})^2}{(1-q^{4n-2})(1-q^{4n+2})}=\prod_{n=1}^{\infty}\frac{(2n)^2}{(2n-1)(2n+1)},
\end{align*}
we have
\begin{align}\label{eq:q-pi2}
\frac{\pi}{2}
=\lim_{q\to 1}\prod_{n=1}^{\infty}\frac{(1-q^{4n})^2}{(1-q^{4n-2})(1-q^{4n+2})}
=\lim_{q\to 1}\frac{1}{4}(1-q^2)q^{-1/2}\theta^2_2(\tau).
\end{align}
Thus, Sun regarded
\begin{align*}
\prod_{n=1}^{\infty}\frac{(1-q^{4n})^2}{(1-q^{4n-1})(1-q^{4n+1})}
\end{align*}
as a $q$-analogue of the Wallis' formula.
And Sun further listed two identities:
\begin{align*}
&\prod_{n=1}^{\infty}\left(\frac{1-q^{2n}}{1-q^{2n-1}}\right)^4=\sum_{k=0}^{\infty}\frac{q^k(1+q^{2k+1})}{(1-q^{2k+1})^2},\\
&\prod_{n=1}^{\infty}\left(\frac{1-q^{2n}}{1-q^{2n-1}}\right)^8=\sum_{k=0}^{\infty}\frac{q^{2k}(1+4q^{2k+1}+q^{4k+2})}{(1-q^{2k+1})^4}.
\end{align*}

Recently, based on results in \cite{AA},  Goswami \cite{AG} generalized Sun's results by obtaining a representation of
\begin{align*}
\prod_{n=1}^{\infty}\left(\frac{1-q^{2n}}{1-q^{2n-1}}\right)^{4k}
\end{align*}
as the sum of Eisenstein series and  cusp forms. In particular, for $k=3$, he proved
\begin{align*}
\sum_{k=0}^{\infty}\frac{q^{k}(1+q^{2k+1})P_{4}(q^{2k+1})}{(1-q^{2k+1})^6}
=256q\prod_{n=1}^{\infty}\left(\frac{1-q^{2n}}{1-q^{2n-1}}\right)^{12}+(q;q)^{12};
\end{align*}
where $P_4(x)=x^4+236x^3+1446x^2+236x+1$.

It seems worthwhile  pointing out  certain important distinctions between our work and  Goswami's. He considered exclusively the decomposition $\theta_2^{4k}(q)$ as the sums of Eisenstein series and  cusp forms for $k=1,2...$; whereas we deal with $\theta_2^{2k}(q)$ in Corollary \ref{cor:theta-8k-q}, Corollary \ref{cor:theta-8k+4-q} and Corollary \ref{cor:theta-8kpm2-q}. In addition, he showed that $\theta_2^{4k}(q)$ (with $q=e^{2\pi i\tau}$) are modular forms of weight $2k$ with respect to the arithmetic group $\Gamma_0(4)$  for $k=1,2...$ In our work, with $q=e^{\pi i\tau}$, we obtain a refinement of Goswami's results by showing that $\theta_2^{8k}(\tau)$ and $\theta_2^{8k-4}(\tau)$ are modular forms of weight $4k$ and $4k-2$ with respect to the arithematic group $\Gamma_0(2)$ and $\Gamma(2)$ for $k=1,2...$, respectively in Section \ref{sec:theta2-4k}. We remark that the arithmetic groups $\Gamma(2)$ and $\Gamma_0(4)$ are isomorphic.

Based on the the corollaries established in Section \sect{lambert series}, we shall provide different and self-contained  proofs of Goswami's results, and re-state them in slightly different forms as Corollary \corol{theta-8k-p} and  Corollary \corol{theta-8k+4-p}. Furthermore, we prove that the palindromic feature of the coefficients of the polynomials exhibited above: $x+1,x^2+4x+1,x^4+236x^3+1446x^2+236x+1$ in the cases of $\theta_2^{4k}(q)$ for $k=1,2$ and $3$. In fact, they hold in general for all the polynomials associated with $\theta_2^{4k}(q) (k=1,2...)$.
The corresponding problem for the cases $\theta_2^{8k\pm2}(q)$ is considerably more complicated than that of $\theta_2^{8k\pm4}(q)$. Not only are the Eisenstein series involved  modular forms with multipliers, but also they are somewhat peculiar combinations of two Eisenstein series.

For the purpose of motivation, we begin with the following list of representations of $\theta_2^{2k}(\tau)$ as the sums of
Eisenstein series and the cusp forms.
\begin{align*}
&\theta_2^2(\tau)=4q^{1/2}\sum_{n=0}^{\infty}\frac{(-1)^nq^{n}}{1-q^{2n+1}},\\
&\theta^4_2(\tau)=16\sum_{n=0}^{\infty}\frac{(2n+1)q^{2n+1}}{1-q^{4n+2}},\\
&\theta_2^6(\tau)=4q^{1/2}\sum_{n=0}^{\infty}(2n+1)^2\left(\frac{q^n}{1+q^{2n+1}}-\frac{(-1)^nq^n}{1-q^{2n+1}}\right),\\
&\theta^8_2(\tau)=2^8\sum_{n=1}^{\infty}\frac{n^3q^{2n}}{1-q^{4n}},\\
&5\theta^{10}_2(\tau)
=4q^{1/2}\sum_{n=0}^\infty(2n+1)^4\left(\frac{ q^{n}}{1+q^{2n+1}}+\frac{(-1)^n q^{n}}{1-q^{2n+1}}\right)
-8q^{1/2}\frac{(q^2;q^2)^{14}_\infty}{(q^4;q^4)^4_\infty},\\
&\theta^{12}_2(\tau)
=16\sum_{n=0}^{\infty}\frac{(2n+1)^5q^{2n+1}}{1-q^{4n+2}}-16q(q^2;q^2)^{12}_\infty,\\
&61\theta^{14}_2(\tau)
=4q^{1/2}\sum_{n=0}^\infty(2n+1)^6\left(\frac{ q^{n}}{1+q^{2n+1}}-\frac{(-1)^n q^{n}}{1-q^{2n+1}}\right)\\
&\qquad\qquad-91\times 2^6q^{3/2}(q^2;q^2)^{10}_\infty(q^4;q^4)^{4}_\infty,\\
&17\theta^{16}_2(\tau)=2^{13}\sum_{n=1}^{\infty}\frac{n^7q^{2n}}{1-q^{4n}}-2^{13}q(q;q)^8_\infty(q^2;q^2)^{8}_\infty,\\
&1385\theta^{18}_2(\tau)
=4q^{1/2}\sum_{n=0}^\infty(2n+1)^8\left(\frac{ q^{n}}{1+q^{2n+1}}+\frac{(-1)^n q^{n}}{1-q^{2n+1}}\right)\\
&\quad\quad\quad\quad-q^{1/2}\frac{(q^2;q^2)^{30}_\infty}{(q^4;q^4)^{12}_\infty}
 -763\times 2^9q^{5/2}(q^2;q^2)^{6}_\infty(q^4;q^4)^{12}_\infty,\\
&31\theta^{20}_2(\tau)
=16\sum_{n=0}^{\infty}\frac{(2n+1)^9q^{2n+1}}{1-q^{4n+2}}-16q\frac{(q^2;q^2)^{28}_\infty}{(q^4;q^4)^{8}_\infty}\\
&\qquad\qquad-77\times2^{12}q^3(q^2;q^2)^{4}_\infty(q^4;q^4)^{16}_\infty,\\
& 50521\theta^{22}_2(\tau)
=4q^{1/2}\sum_{n=0}^\infty(2n+1)^{10}\left(\frac{ q^{n}}{1+q^{2n+1}}-\frac{(-1)^n q^{n}}{1-q^{2n+1}}\right)\\
&\quad\quad\quad\quad-138677\times2^{14}q^{7/2}(q^2;q^2)^{2}_\infty(q^4;q^4)^{20}_\infty
 -7381\times2^6q^{3/2}\frac{(q^2;q^2)^{26}_\infty}{(q^4;q^4)^{4}_\infty},\\
&691\theta^{24}_2(\tau)=2^{16}\sum_{n=1}^{\infty}\frac{n^{11}q^{2n}}{1-q^{4n}}-2^{16}q(q;q)^{24}_\infty
-259\times2^{19}q^2(q^2;q^2)^{24}_\infty.\\
\end{align*}
We observe that, among these examples, the Eisenstein series part of $\theta^{2k}_2(\tau)$ can be characterized in terms of the congruence class modulo 8 to which $2k$ belongs. For example, if $2k\equiv 4 \mod 8$, the Eisenstein series is, modulo the constant multiple, of the form
\begin{align*}
\sum_{n=0}^{\infty}\frac{(2n+1)^{k-1}q^{2n+1}}{1-q^{4n+2}}.
\end{align*}

We now describe the setting and methods for establishing the identities of this paper.
Let
\begin{align*}
\Gamma_0(2)=\left\{ \left(\begin{array}{cc}
a&b\\
c&d
\end{array}\right)~\bigg|~ ad-bc=1, c\equiv 0 \pmod 2\right\}
\end{align*}
and
\begin{align*}
T=\left(\begin{array}{cc}
1&1\\
0&1
\end{array}\right)
\quad\text{and}\quad
S=\left(\begin{array}{cc}
1&0\\
-2&1
\end{array}\right).
\end{align*}

For $\tau\in\mathbb H$ and $
\sigma=\left(\begin{array}{cc}
a&b\\
c&d
\end{array}\right)\in \Gamma_0(2),
$
define $\sigma\tau=\frac{a\tau+b}{c\tau+d}.$ Let
\begin{align*}
G_0(2)=\{\sigma\tau ~|~ \sigma=\left(\begin{array}{cc}
a&b\\
c&d
\end{array}\right)\in \Gamma_0(2)\}.
\end{align*}
Then
\begin{align*}
T(\tau)=\tau+1\quad\text{and}\quad S\tau=\frac{\tau}{-2\tau+1}
\end{align*}
are generators for $G_0(2)$.
A choice of fundamental domain of $G_0(2)$ is
\begin{align*}
\Omega_0(2)=\{\tau ~|~ 0\le x\le 1 \text{ and } |\tau-1/2|\ge 1/2\}
\end{align*}
with the boundary properly identified.  There are two cusp points: 0 and $\infty$.\\

Using the facts:
\begin{align*}
\theta_2(\tau+1)=e^{\pi{ i}/4}\theta_2(\tau),\quad \theta_2(-1/\tau)=\sqrt{-i\tau}\theta_4(\tau)
\quad \text{and}\quad \theta_4(-1/\tau)=\sqrt{-i\tau}\theta_2(\tau),
\end{align*}
it follows that
\begin{align*}
\theta^2_2(\tau+1)={ i}\theta^2_2(\tau)\quad\text{and}\quad
\theta_2^2\left(\frac{\tau}{-2\tau+1}\right)=
(-2\tau+1) \theta_2^2\left(\tau\right).
\end{align*}
In general,  for $\sigma=\left(\begin{array}{cc}
a&b\\
c&d
\end{array}\right)\in\Gamma_0(2) $,  we have
\begin{align}\label{theta-2k-tran}
\theta_2^{2}\left(\frac{a\tau+b}{c\tau+d}\right)=\psi(\sigma)
(c\tau+d)\theta_2^{2}\left(\tau\right),
\end{align}
where
\begin{align*}
\psi(\sigma)= \begin{cases} -{ i}e^{\frac{\pi{ i}}{2}(bd+1)} \quad & \text{if }  c\equiv2\pmod 4, \\[0.1in]
       e^{\frac{\pi{ i}}{2}(bd+d-1)} \quad & \text{if } c\equiv0\pmod 4,
\end{cases}
\end{align*}
and the details is given in Theorem \thm{theta2-trans}.
Thus, in general,
\begin{align*}
\theta_2^{2k}\left(\frac{a\tau+b}{c\tau+d}\right)=\psi^k(\sigma)
(c\tau+d)^k\theta_2^{2k}\left(\tau\right)
\end{align*}
and, in particular,
\begin{align*}
\theta_2^{8k}\left(\frac{a\tau+b}{c\tau+d}\right)=
(c\tau+d)^{4k}\theta_2^{8k}\left(\tau\right).
\end{align*}

Let $k$ be a non-negative integer and let
$\mathbf{M}_{k}(\Gamma_0(2);\psi)$
denote the space of functions $f$ on $\mathbb H$
satisfying (see \cite[p. 78]{TA}):

(1) $f$ is analytic on $\mathbb H$;

(2) $f$ is analytic at all cusps of $\Gamma_0(2)$;

(3)
\begin{align*}
f\left(\frac{a\tau+b}{c\tau+d}\right)=\psi(\sigma)(c\tau+d)^{k} f(\tau);
\end{align*}
where $\sigma=\left(\begin{array}{cc}
a&b\\
c&d
\end{array}\right)\in\Gamma_0(2)$.

Each $M_{k}(\Gamma_0(2);\psi)$
can be decomposed further as:
\begin{align*}
M_{k}(\Gamma_0(2);\psi)= \mathbf E_{k}(\Gamma_0(2);\psi) \oplus\mathbf S_{k}(\Gamma_0(2);\psi);
\end{align*}
where $\mathbf E_{k}(\Gamma_0(2);\psi)$ and $\mathbf S_{k}(\Gamma_0(2);\psi)$ are  the vector spaces of the cusp forms and Eisenstein series of weight $k$.
Thus,
\begin{align*}
\theta_2^{2k}\left(\tau\right)\in M_{k}(\Gamma_0(2);\psi^k).
\end{align*}
We note that if $k\equiv 0$ mod 4, then $\psi^k\equiv 1$.  We denote $\mathbf S_{k}(\Gamma_0(2))=\mathbf S_{k}(\Gamma_0(2);1)$, $\mathbf E_{k}(\Gamma_0(2))=\mathbf E_{k}(\Gamma_0(2);1)$ and $\mathbf M_{k}(\Gamma_0(2))=\mathbf M_{k}(\Gamma_0(2);1)$.\\


Let  $\wp(z|\tau)$ denote the Weierstrass elliptic function with  the period lattice $\Lambda$ generated by $\pi$ and $\pi\tau$. We say $z$ is a point of order $N$ if $Nz\in\Lambda$. We will be concerned only with the values of the Weierstrass elliptic function $\wp(z|\tau)$ at certain points of orders 2 and 4. In Section \sect{pre}, we will list the key identities of the theta functions, the formulas expressing $\wp(z|\tau)$ in terms of the theta functions,
 and the values of $\wp(z|\tau)$ along with its derivatives at certain points of orders  2 and 4. In Section \sect{lambert series}, we derive  the formulas expressing $\wp^{(n)}(z|\tau)$ in terms of the Eisenstein series and $q$-series expansions. In Sections \sect{theta2-4k} and \sect{theta2-2k}, we decompose $\theta_2^{2k}(\tau)$ into the sum of Eisenstein series and  cusp forms. 
 This is accomplished by constructing Eisenstein series which matches the values of $\theta_2^{2k}(\tau)$ at the cusp points of $\Omega_0(2)$  using the Weierstrass elliptic function. In Section \sect{recurrence},  we derive a recurrence relation expressing the value of $\wp^{(n)}(z|\tau)$ in terms of $\wp^{(k)}(z|\tau)$, $k< n$. Appealing to the recurrence relation, we derive the identities listed at the beginning of this section. Some miscellaneous results of independent interest and a list of identities are mentioned in Sections \sect{comment} and \sect{list}.\\
\section{preliminaries}\label{sec:pre}

We  list the needed identities of the theta functions and the Weierstrass elliptic function.

(A) From \cite{whittaker and watson}, we find
\begin{align*}
&\theta^4_2(\tau)=\theta^4_3(\tau)-\theta^4_4(\tau),
\quad\quad\theta_2(\tau)\theta_3(\tau)\theta_4(\tau)=2q^{1/4}(q^2;q^2)^3_\infty,\\
&\theta_2(\tau)\theta_3(\tau)=\frac 12\theta^2_2(\tau/2),
\quad\quad\theta_3(\tau)\theta_4(\tau)=\theta^2_4(2\tau),
\end{align*}
and
\begin{align} \label{eq:lim-theta2-i}
\lim_{\tau\rightarrow 0}\tau\theta^2_2(\tau)={ i}.
\end{align}

(B) From \cite[Theorem 1.12, Theorem 1.14, Theorem 1.18]{TA}, we have
 \begin{align*}
(\wp^{\prime})^2=4\wp^3-g_2\wp-g_3=4(\wp-e_1)(\wp-e_2)(\wp-e_3),
\end{align*}
where
\begin{align*}
&g_2=\frac43\left(1+240\sum_{n=1}^{\infty}\frac{n^3q^{2n}}{1-q^{2n}}\right),\quad\quad\quad
g_3=\frac{8}{27}\left(1-504\sum_{n=1}^{\infty}\frac{n^5q^{2n}}{1-q^{2n}}\right),\\
&e_1=\wp\left(\frac{\pi}{2}|\tau\right)=\frac13\left(\theta^4_3(\tau)+\theta^4_4(\tau)\right),\quad\quad
e_2=\wp\left(\frac{\pi\tau}{2}|\tau\right)=-\frac13\left(\theta^4_2(\tau)+\theta^4_3(\tau)\right),\\
&e_3=\wp\left(\frac{\pi+\pi\tau}{2}|\tau\right)=\frac13\left(\theta^4_2(\tau)-\theta^4_4(\tau)\right);
\end{align*}

From \cite[p. 10]{TA}, we find
\begin{align*}
\wp(z|\tau)=\frac{1}{z^2}+\sum_{\substack{
m,n=-\infty\\(m,n)\ne(0,0)}}^{\infty}\frac{1}{(z+m\pi+n\pi\tau)^2}-\frac{1}{(m\pi+n\pi\tau)^2},
\end{align*}

\begin{align}\label{eq:defn-wp-k}
\wp^{(k)}(z|\tau)=(k+1)!\sum_{m,n=-\infty}^{\infty}\frac{1}{(z+m\pi+n\pi\tau)^{k+2}},
\end{align}
for all positive integers $k$.

(C) The connection between the theta function $\theta_1(z|\tau)$ and $\wp(z|\tau)$ is given by
\begin{align*}
\wp(z|\tau)=-\left(\frac{\theta^{\prime}_1}{\theta_1}\right)^{\prime}(z|\tau)-\frac{E_2(\tau)}{3},
\end{align*}
where
\begin{align*}
E_2(\tau)=1-24\sum_{n=1}^{\infty}\frac{nq^{2n}}{1-q^{2n}}.
\end{align*}

(D) We recall that \cite[p. 489]{whittaker and watson}
\begin{align*}
&\frac{\theta^{\prime}_1}{\theta_1}(z|\tau)=\cot z+4\sum_{n=1}^{\infty}\frac{q^{2n}}{1-q^{2n}}\sin {2nz},\\
&\frac{\theta^{\prime}_4}{\theta_4}(z|\tau)=4\sum_{n=1}^{\infty}\frac{q^{n}}{1-q^{2n}}\sin {2nz},\\
& \frac{\theta^{\prime}_1}{\theta_1}(z+\frac{\pi\tau}{2}|\tau)=\frac{\theta^{\prime}_4}{\theta_4}(z|\tau)-{ i}.
\end{align*}
Thus
\begin{align}\label{eq:theta4-prime}
\wp(z+\frac{\pi\tau}{2}|\tau)=-\left(\frac{\theta^{\prime}_4}{\theta_4}\right)^\prime(z|\tau)-\frac{E_2(\tau)}{3}.
\end{align}
Setting $z=0$ and $z=\frac{\pi}{2}$ in the above equation, respectively, we derive \cite{Berndt}\cite{Liu-Rama}
\begin{align}
&\wp\left(\frac{\pi\tau}{2}|\tau\right)
=-\frac13\left(1+24\sum_{n=0}^{\infty}\frac{nq^n}{1+q^n}\right)=-\frac13(\theta^4_2(\tau)+\theta^4_3(\tau)),\nonumber\\
&\wp\left(\frac{\pi\tau}{2}|\tau\right)-\wp\left(\frac{\pi+\pi\tau}{2}|\tau\right)
=-2^4\sum_{n=0}^{\infty}\frac{(2n+1)q^{2n+1}}{1-q^{4n+2}}=-\theta^4_2(\tau)\label{eq:wp-pi-k=0}.
\end{align}

(E) From \cite{Liu-Rama}, we find
\begin{align*}
&\wp\left(\frac{\pi+2\pi\tau}{4}|\tau\right)
=-\frac13(\theta^4_3(2\tau)-5\theta^4_2(2\tau)),\\
&\wp^{\prime}\left(\frac{\pi+2\pi\tau}{4}|\tau\right)
=4\theta^2_2(2\tau)\theta^4_4(2\tau),\\
&\wp^{\prime\prime}\left(\frac{\pi+2\pi\tau}{4}|\tau\right)
=-2^4\theta^4_2(2\tau)\theta^4_4(2\tau).
\end{align*}
Replacing $\tau$ by $\tau+1/2$ in the above equations, we have
\begin{align*}
&\wp\left(\frac{\pi+\pi\tau}{2}|\frac{2\tau+1}{2}\right)
=-\frac13(\theta^4_4(2\tau)+5\theta^4_2(2\tau)),\\
&\wp^{\prime}\left(\frac{\pi+\pi\tau}{2}|\frac{2\tau+1}{2}\right)
=4{ i}\theta^2_2(2\tau)\theta^4_3(2\tau),\\
&\wp^{\prime\prime}\left(\frac{\pi+\pi\tau}{2}|\frac{2\tau+1}{2}\right)
=2^4\theta^4_2(2\tau)\theta^4_3(2\tau).
\end{align*}

(F) Since
$g_2=-4(e_1e_2+e_1e_3+e_2e_3)
=\frac 43\left(\theta^8_3(\tau)+\theta^8_2(\tau)-\theta^4_2(\tau)\theta^4_3(\tau)\right)$
 and
$\wp^{''}=6\wp^2-\frac12g_2$,
we obtain
\begin{align*}
&\wp^{''}\left(\frac{\pi}{2}|\tau\right)
=6\wp^{2}\left(\frac{\pi}{2}|\tau\right)-\frac12g_2
=2\theta^4_3(\tau)\theta^4_4(\tau),\\
&\wp^{''}\left(\frac{\pi\tau}{2}|\tau\right)
=6\wp^{2}\left(\frac{\pi\tau}{2}|\tau\right)-\frac12g_2
=2\theta^4_2(\tau)\theta^4_3(\tau)=\frac{1}{8}\theta^8_2(\tau/2),\\
&\wp^{''}\left(\frac{\pi+\pi\tau}{2}|\tau\right)
=6\wp^{2}\left(\frac{\pi+\pi\tau}{2}|\tau\right)-\frac12g_2
=-2\theta^4_2(\tau)\theta^4_4(\tau).
\end{align*}

\section{Some Eisenstein series derived from the Weierstrass Elliptic function}
\label{sec:lambert series}
We now construct Eisenstein series which will play a vital role in the decomposition of $\theta_2^{2n}(\tau)$ into sums of
Eisenstein series and cusp forms.

Let
\begin{align*}
\chi_2(n)=\begin{cases}
0\quad\text{if $n$ is even}\\
1\quad\text{if $n$ is odd}
\end{cases}
\end{align*}
and $\chi(n)=\sin\left(\frac{n\pi}{2}\right)$. Define
\begin{align*}
\sigma_{k,\chi}(n):=\sum_{d|n}d^k\chi(d)=\sum_{d|n}d^k\sin\left(\frac{d\pi}{2}\right).
\end{align*}

Differentiating \eqn{theta4-prime} repeatedly, we derive

\begin{lemma}\label{lem:wp-k-q}
For all integers $k\ge 1$, we have
\begin{align}
\wp^{(2k)}\left(z+\frac{\pi\tau}{2}|\tau\right)
=(-1)^{k+1}2^{2k+3}\sum_{n=1}^{\infty}\frac{n^{2k+1}q^{n}}{1-q^{2n}}\cos 2nz,\label{eq:wp-(2k)-pitau}\\
\wp^{(2k-1)}\left(z+\frac{\pi\tau}{2}|\tau\right)
=(-1)^{k+1}2^{2k+2}\sum_{n=1}^{\infty}\frac{n^{2k}q^{n}}{1-q^{2n}}\sin 2nz.\label{eq:wp-(2k-1)-pitau}
\end{align}
\end{lemma}

Taking $z=0$ and $z=\pi/2$ in \eqn{wp-(2k)-pitau}, respectively, and combining with \eqn{defn-wp-k}, then
\begin{theorem}
For all integers $k\geq1$, we have
\begin{align}
 &\wp^{(2k)}\left(\frac{\pi\tau}{2}|\tau\right)\nonumber\\
=&\frac{(2k+1)!}{\pi^{2k+2}}\sum_{m,n=-\infty}^{\infty}\frac{\chi_2(n)}{(m+n\frac{\tau}{2})^{2k+2}}
\label{eq:wp-pitau-2-tau}\\
=&(-1)^{k+1}2^{2k+3}\sum_{n=1}^{\infty}\frac{n^{2k+1}q^n}{1-q^{2n}},
\label{eq:wp-pitau-2-q}\\
 &\wp^{(2k)}\left(\frac{\pi\tau}{2}|\tau\right)-\wp^{(2k)}\left(\frac{\pi+\pi\tau}{2}|\tau\right)\nonumber\\
=&\frac{2^{2k+2}(2k+1)!}{\pi^{2k+2}}\sum_{m,n=-\infty}^{\infty}\frac{(-1)^m\chi_2(n)}{(m+n\tau)^{2k+2}}
\label{eq:wp-pi-2-tau}\\
=&(-1)^{k+1}2^{2k+4}\sum_{n=0}^{\infty}\frac{(2n+1)^{2k+1}q^{2n+1}}{1-q^{4n+2}}.
\label{eq:wp-pi-2-q}
\end{align}
\end{theorem}

\begin{theorem}\label{thm:wp-2k-1} For all integers $k\geq1$, we have
\begin{align}
&\wp^{(2k-1)}\left(\frac{\pi+2\pi\tau}{4}|\tau\right)-{i}\wp^{(2k-1)}\left(\frac{\pi+\pi\tau}{2}|\frac{2\tau+1}{2}\right)\nonumber\\
=&\frac{2^{4k+2}(2k)!}{\pi^{2k+1}}\sum_{m,n=-\infty}^{\infty}\frac{1}{(4m+1+(4n+2)\tau)^{2k+1}}
-\frac{{ i}}{(4m+2n+2+(4n+2)\tau)^{2k+1}}\mylabel{eq:wp-2k-1-(-i)-tau}\\
=&(-1)^{k}2^{2k+2}\sum_{n=0}^{\infty}(2n+1)^{2k}\left(\frac{(-1)^n q^{2n+1}}{1-q^{4n+2}}
+\frac{ q^{2n+1}}{1+q^{4n+2}}\right)\mylabel{eq:wp-2k-1-(-i)-q}\\
=&(-1)^{k+1}2^{2k+3}\sum_{n=0}^{\infty}\sigma_{2k,\chi}(4n+1)q^{4n+1};\mylabel{eq:wp-2k-1-(-i)-sigma}\\
&\wp^{(2k-1)}\left(\frac{\pi+2\pi\tau}{4}|\tau\right)
+{ i}\wp^{(2k-1)}\left(\frac{\pi+\pi\tau}{2}|\frac{2\tau+1}{2}\right)\nonumber\\
=&\frac{2^{4k+2}(2k)!}{\pi^{2k+1}}\sum_{m,n=-\infty}^{\infty}\frac{1}{(4m+1+(4n+2)\tau)^{2k+1}}
+ \frac{ i}{(4m+2n+2+(4n+2)\tau)^{2k+1}}\mylabel{eq:wp-2k-1-(i)-tau}\\
=&(-1)^{k}2^{2k+2}\sum_{n=0}^{\infty}(2n+1)^{2k}
\left(\frac{(-1)^n q^{2n+1}}{1-q^{4n+2}}-\frac{q^{2n+1}}{1+q^{4n+2}}\right)\mylabel{eq:wp-2k-1-(i)-q}\\
=&(-1)^{k+1}2^{2k+3}\sum_{n=0}^{\infty}\sigma_{2k,\chi}(4n+3)q^{4n+3}.\label{eq:wp-2k-1-(i)-sigma}
\end{align}
\end{theorem}

\begin{proof}
From \eqn{defn-wp-k} and Lemma \lem{wp-k-q}:
\begin{align*}
 &\wp^{(2k-1)}\left(\frac{\pi+2\pi\tau}{4}|\tau\right)\\
=&(-1)^{k}2^{2k+2}\sum_{n=1}^{\infty}\frac{(-1)^n(2n+1)^{2k}q^{2n+1}}{1-q^{4n+2}}\\
=&\frac{2^{4k+2}(2k)!}{\pi^{2k+1}}\sum_{m,n=-\infty}^{\infty}\frac{1}{(4m+1+(4n+2)\tau)^{2k+1}},\\
  &-{ i}\wp^{(2k-1)}\left(\frac{\pi+\pi\tau}{2}|\frac{2\tau+1}{2}\right)\\
=&(-1)^{k}2^{2k+2}\sum_{n=1}^{\infty}\frac{(2n+1)^{2k}q^{2n+1}}{1+q^{4n+2}}\\
=&\frac{2^{4k+2}(2k)!}{\pi^{2k+1}}\sum_{m,n=-\infty}^{\infty}\frac{-{ i}}{(4m+2n+2+(4n+2)\tau)^{2k+1}}.
\end{align*}

The identities \eqn{wp-2k-1-(-i)-tau}, \eqn{wp-2k-1-(-i)-q},\eqn{wp-2k-1-(i)-tau} and \eqn{wp-2k-1-(i)-q} follow readily.

For \eqn{wp-2k-1-(-i)-sigma} and \eqn{wp-2k-1-(i)-sigma},  by expanding the summands into  geometric series and inverting the order of summation, we readily find that
\begin{align*}
 &\sum_{n=0}^{\infty}\frac{(-1)^{n}(2n+1)^{k}q^{2n+1}}{1-q^{4n+2}}\\
=&\sum_{n=1}^{\infty}\frac{n^k q^n}{1-q^{2n}}\sin \frac{n\pi}{2}\\
=&\sum_{n=1,2\nmid n}^{\infty}\frac{n^k q^n}{1-q^{2n}}\sin \frac{n\pi}{2}\\
=&\sum_{n=0}^{\infty}\sigma_{k,\chi}(2n+1)q^{2n+1}.
\end{align*}
Letting $z=\frac{\pi}{4}$ in \eqn{wp-(2k-1)-pitau}, we have
\begin{align*}
 &\wp^{(2k-1)}\left(\frac{\pi+2\pi\tau}{4}|\tau\right)\\
=&(-1)^{k+1}2^{2k+2}\sum_{n=1}^{\infty}\frac{n^{2k}q^n}{1-q^{2n}}\sin\frac{n\pi}{2}\\
=&(-1)^{k+1}2^{2k+2}\sum_{n=0}^{\infty}\sigma_{2k,\chi}(2n+1)q^{2n+1}.
\end{align*}
Replacing $\tau$ by $\tau+1/2$, we have
\begin{align*}
 &-{ i}\wp^{(2k-1)}\left(\frac{\pi+\pi\tau}{2}|\frac{2\tau+1}{2}\right)\\
=&(-1)^{k}2^{2k+2}\sum_{n=1}^{\infty}\frac{i^{n+1}n^{2k}q^n}{1+(-1)^{n+1}q^{2n}}\sin\frac{n\pi}{2}\\
=&(-1)^{k+1}2^{2k+2}\sum_{n=0}^{\infty}(-1)^n\sigma_{2k,\chi}(2n+1)q^{2n+1}.
\end{align*}
Combining the above two equations, we get the desired identities.
\end{proof}

For later use, we define, for integer $k\ge 1$,

\begin{align*}
M_{4k}(\tau;\chi_2)
:=&\sum_{m,n=-\infty}^{\infty}\frac{\chi_2(n)}{(m+n\tau)^{4k}},\\
M^*_{4k+2}(\tau;\chi_2)
:=&\sum_{m,n=-\infty}^{\infty}\frac{(-1)^m\chi_2(n)}{(m+n\tau)^{4k+2}},\\
M_{2k+1}(\tau;\chi)
:=&\sum_{m,n=-\infty}^{\infty}\frac{1}{(4m+1+(2n+1)\tau)^{2k+1}}\\
&+\sum_{m,n=-\infty}^{\infty}\frac{(-1)^{k+1}{ i}}{(4m+2n+2+(2n+1)\tau)^{2k+1}}.
\end{align*}
It is interesting to point out that $M_{2k+1}(\tau;\chi)$ is  the difference of Eisenstein series constructed from elliptic functions of different moduli, $\tau, \tau+\frac12$, evaluated at points of order 4 .\\

Recall
\begin{align*}
&\lim_{\tau\rightarrow 0}\tau\theta^2_2(\tau)={ i},\qquad\quad\qquad\qquad
\lim_{\tau\rightarrow \infty}\tau\theta^2_2(\tau)=0,\\
&\theta^{2k}_2(\tau+1)={ i}^k\theta^{2k}_2(\tau),\qquad
\theta_2^{2k}\left(\frac{\tau}{-2\tau+1}\right)=(-2\tau+1)^k \theta_2^{2k}\left(\tau\right).
\end{align*}
To decompose  $\theta_2^{2k}(\tau)$ into the sum of  Eisenstein series and  cusp forms, we will show, in the following sections, that the behavior of these Eisenstein series under the the transformations:
\begin{align*}
T(\tau)=\tau +1 \quad\text{and}\quad S(\tau)=\frac{\tau}{-2\tau+1}
\end{align*}
matches precisely with that of $\theta_2^{2k}(\tau)$ under the same transformations. We recall that $\Gamma_0(2)$, in the fundamental domain $\Omega_0(2)$, has two cusp points with 0 and  $\infty$.  Since both $\theta_2^{2k}(\tau)$ and these Eisenstein series vanish as $\tau\to\infty$, we will scale these Eisenstein series to match the values of the corresponding  $\theta_2^{2k}(\tau)$  at 0. This allows us to decompose  $\theta_2^{2k}(\tau)$ into the sum of  Eisenstein series and  cusp forms.

\section{representations of $\theta^{8k}_2(\tau)$ and  $\theta^{8k-4}_{2}(\tau)$}\label{sec:theta2-4k}

We begin by recalling  \cite[p. 35 (21)]{AE}, for all positive integers $k$,
\begin{align}\mylabel{eq:zeta-2k}
\zeta(2k)=(-1)^{k+1}2^{2k-1}\pi^{2k}\frac{B_{2k}}{(2k)!},
\end{align}
 where $B_{2k}$ denotes the $2k$-th Bernoulli number.\\

\subsection{Representations of $\theta^{8k}_2(\tau)$}\label{sec:theta2-8k}

Recall
\begin{align*}
\chi_2(n)=\begin{cases}
0\quad\text{if $n$ is even}\\
1\quad\text{if $n$ is odd}.
\end{cases}
\end{align*}

\begin{theorem}\label{thm:theta2-8k} For all positive integers $k$, we obtain
\begin{align*}
\theta^{8k}_2(\tau)=\frac{(4k)!}{\pi^{4k}(1-2^{4k})B_{4k}}\sum_{m,n=-\infty}^{\infty}\frac{\chi_2(n)}{(m+n\tau)^{4k}}+T_{4k}(\tau);
\end{align*}
where $T_{4k}(\tau)\in \mathbf S_{4k}(\Gamma_0(2))$.
\end{theorem}

\begin{proof}

It is easy to verify that, by checking the generators of $\Gamma_0(2)$, $\theta^{8k}_2(\tau)$ and
\begin{align*}
M_{4k}(\tau;\chi_2):=\sum_{m,n=-\infty}^{\infty}\frac{\chi_2(n)}{(m+n\tau)^{4k}}
\end{align*}
are modular of weight $4k$ with respect to the arithmetic group $\Gamma_0(2)$ which has fundamental domain $\Omega_0(2)$ with cusps at 0 and $\infty$.

We now consider the values of $M_{4k}(\tau;\chi_2)$ and $\theta^{8k}_2(\tau)$ at the cusps.
Clearly, as $\tau\rightarrow \infty$,
\begin{align*}
M_{4k}(\tau;\chi_2)\rightarrow0  \quad\quad   and \quad\quad  \theta^{8k}_2(\tau)\rightarrow 0.
\end{align*}
We need  to choose an appropriate constant $A_{4k}$, such that as $\tau\rightarrow0$,
\begin{align}\label{eq:A4k}
\lim_{\tau\rightarrow0}\tau^{4k}\left(\theta^{8k}_2(\tau)-A_{4k}M_{4k}(\tau;\chi_2)\right)=0.
\end{align}
This will imply that
\begin{align}\label{eq:defn-A4k}
\theta^{8k}_2(\tau)=A_{4k}M_{4k}(\tau;\chi_2)+T_{4k}(\tau);
\end{align}
where $T_{4k}(\tau)\in \mathbf S_{k}(\Gamma_0(2))$.

We compute the values of $A_{4k}$.

From \eqn{lim-theta2-i} and \eqn{defn-A4k},
\begin{align}\label{eq:lim-8k}
1=\lim_{\tau\rightarrow0}\tau^{4k}\theta^{8k}_2(\tau)=A_{4k}\lim_{\tau\rightarrow 0}\tau^{4k}M_{4k}(\tau;\chi_2).
\end{align}

We note
\begin{align*}
 &\sum_{m,n=-\infty}^{\infty}\frac{\chi_2(n)}{(m+n\tau)^{4k}}\\
=&\sum_{m,n=-\infty}^{\infty}\frac{1}{\left(m+(2n+1)\tau\right)^{4k}}\\
=&\frac{1}{\tau^{4k}}\sum_{n=-\infty}^{\infty}\frac{1}{(2n+1)^{4k}}
+\sum_{\substack{m,n=-\infty\\m\neq0}}^{\infty}\frac{1}{\left(m+(2n+1)\tau\right)^{4k}}.
\end{align*}

Next, we observe
\begin{align*}
 &\sum_{n=-\infty}^{\infty}\frac{1}{(2n+1)^{4k}}\\
=&2\sum_{n=0}^{\infty}\frac{1}{(2n+1)^{4k}}
=2\sum_{n=1}^{\infty}\frac{1}{n^{4k}}-2\sum_{n=1}^{\infty}\frac{1}{(2n)^{4k}}\\
=&2\sum_{n=1}^{\infty}\frac{1}{n^{4k}}-2^{1-4k}\sum_{n=1}^{\infty}\frac{1}{n^{4k}}=(2-2^{1-4k})\zeta(4k).
\end{align*}

Then,
\begin{align}
&\lim_{\tau\rightarrow0}\tau^{4k}\sum_{m,n=-\infty}^{\infty}\frac{1}{\left(m+(2n+1)\tau\right)^{4k}}\nonumber\\
=&\sum_{n=-\infty}^{\infty}\frac{1}{(2n+1)^{4k}}+\lim_{\tau\rightarrow0}\tau^{4k}\sum_{\substack{m,n=-\infty\\m\neq0}}^{\infty}\frac{1}{\left(m+(2n+1)\tau\right)^{4k}}\nonumber\\
=&(2-2^{1-4k})\zeta(4k)+\lim_{\tau\rightarrow0}\sum_{\substack{m,n=-\infty\\m\neq0}}^{\infty}\frac{1}{\left(m/\tau+(2n+1)\right)^{4k}}\nonumber\\
=&(2-2^{1-4k})\zeta(4k)\label{eq:lim-A4k}.
\end{align}

From \eqn{lim-8k}  and \eqn{lim-A4k}, we obtain
\begin{align}\label{eq:B-4k}
A_{4k}=\frac{(4k)!}{(1-2^{4k})\pi^{4k}B_{4k}}.
\end{align}
Substituting \eqn{B-4k} into  the equation \eqn{defn-A4k}, we complete the proof of Theorem \thm{theta2-8k}.

\end{proof}

\begin{cor}\label{cor:theta-8k-q} For all positive integers $k$, we have
\begin{align*}
\theta^{8k}_2(\tau)=\frac{2^{4k+3}k}{(1-2^{4k})B_{4k}}\sum_{n=1}^{\infty}\frac{n^{4k-1}q^{2n}}{1-q^{4n}}+T_{4k}(\tau).
\end{align*}
\end{cor}

\begin{proof}
From \eqn{wp-pitau-2-tau} and \eqn{wp-pitau-2-q}, we find
\begin{align}
 &\frac{(4k-1)!}{\pi^{4k}}\sum_{m,n=-\infty}^{\infty}\frac{\chi_2(n)}{(m+n\tau)^{4k}}\nonumber\\
=&\frac{(4k-1)!}{\pi^{4k}}\sum_{m,n=-\infty}^{\infty}\frac{1}{\left(m+(2n+1)\tau\right)^{4k}}\nonumber\\
=&\wp^{(4k-2)}(\pi\tau|2\tau)\nonumber\\
=&2^{4k+1}\sum_{n=1}^{\infty}\frac{n^{4k-1}q^{2n}}{1-q^{4n}}\label{eq:sum-4k}.
\end{align}
From \eqn{sum-4k} and Theorem \thm{theta2-8k}, we get the desired identity.
\end{proof}

\subsection{Representations of $\theta^{8k+4}_2(\tau)$}\label{theta2-8k+4}
\begin{theorem}\label{thm:theta2-8k+4} For all positive integers $k$, we have
\begin{align}
\theta^{8k+4}_2(\tau)
=\frac{(4k+2)!}{(1-2^{4k+2})\pi^{4k+2}B_{4k+2}}
\sum_{m,n=-\infty}^{\infty}\frac{(-1)^m\chi_2(n)}{(m+n\tau)^{4k+2}}+ T_{4k+2}(\tau)\label{eq:theta2-8k+4};
\end{align}
where $T_{4k+2}(\tau)\in \mathbf S_{4k+2}(\Gamma_0(2);\psi^2)$.
\end{theorem}

\begin{proof}
Let
\begin{align*}
M^*_{4k+2}(\tau;\chi_2)=\sum_{m,n=-\infty}^{\infty}\frac{(-1)^m\chi_2(n)}{(m+n\tau)^{4k+2}}.
\end{align*}

We need to show
\begin{align*}
&M^*_{4k+2}(\tau+1;\chi_2)=-M^*_{4k+2}(\tau;\chi_2),\\
&M^*_{4k+2}\left(\frac{\tau}{-2\tau+1};\chi_2\right)=(-2\tau+1)^{4k+2}M^*_{4k+2}(\tau;\chi_2).
\end{align*}
However, we will establish the following stronger lemma.
\begin{lemma}
Let $\sigma=\left(\begin{array}{cc}
a&b\\
c&d
\end{array}\right)\in\Gamma_0(2)$. We have

\begin{align*}
M^*_{4k+2}\left(\frac{a\tau+b}{c\tau+d};\chi_2\right)
=(-1)^b(c\tau+d)^{4k+2}M^*_{4k+2}(\tau;\chi_2).
\end{align*}
\end{lemma}

\begin{proof}

Since $c\equiv0\pmod2$ and $ad\equiv1\pmod2$, thus $a\equiv d\equiv 1\pmod2$,
\begin{align*}
&\chi_2(mc+na)=\chi_2(n)\chi_2(a)=\chi_2(n),\\
&(-1)^{md+nb}=(-1)^m(-1)^{nb},
\end{align*}
and, since $\chi_2(n)=0$ if $n$ is even,
\begin{align*}
(-1)^{md+nb}\chi_2(mc+na)=(-1)^{m}(-1)^{nb}\chi_2(n)=(-1)^m(-1)^b\chi_2(n).
\end{align*}
Then,
\begin{align*}
 &M^*_{4k+2}\left(\frac{a\tau+b}{c\tau+d};\chi_2\right)\\
=&(c\tau+d)^{4k+2}\sum_{m,n=-\infty}^{\infty}\frac{(-1)^m\chi_2(n)}{\left((md+nb)+(mc+na)\tau\right)^{4k+2}}\\
=&(-1)^b(c\tau+d)^{4k+2}\sum_{m,n=-\infty}^{\infty}\frac{(-1)^{md+nb}\chi_2(mc+na)}{\left((md+nb)+(mc+na)\tau\right)^{4k+2}}\\
=&(-1)^b(c\tau+d)^{4k+2}M^*_{4k+2}(\tau;\chi_2).
\end{align*}
This establishes the desired identity.
\end{proof}

We need  choose an appropriate constant $A_{4k+2}$, such that as $\tau\rightarrow0$,
\begin{align}\label{eq:A-4k+2}
\lim_{\tau\rightarrow0}\tau^{4k+2}\left(\theta^{8k+4}_2(\tau)-A_{4k+2}M_{4k+2}(\tau;\chi_2)\right)=0.
\end{align}

Next, we compute the values of $A_{4k+2}$.

From \eqn{lim-theta2-i} and \eqn{A-4k+2}, we have
\begin{align}\label{eq:lim-8k+4}
-1=\lim_{\tau\rightarrow0}\tau^{4k+2}\theta^{8k+4}_2(\tau)
=A_{4k+2}\lim_{\tau\rightarrow 0}\tau^{4k+2}M^*_{4k+2}(\tau;\chi_2).
\end{align}

Then,
\begin{align}
&\lim_{\tau\rightarrow0}\tau^{4k+2}\sum_{m,n=-\infty}^{\infty}\frac{(-1)^m\chi_2(n)}{\left(m+n\tau\right)^{4k+2}}\nonumber\\
=&\sum_{n=-\infty}^{\infty}\frac{1}{(2n+1)^{4k+2}}
 +\lim_{\tau\rightarrow0}\sum_{\substack{m,n=-\infty\\m\neq0}}^{\infty}\frac{(-1)^m\chi_2(n)}{\left(m/\tau+n\right)^{4k+2}}\nonumber\\
=&(2-2^{-4k-1})\zeta(4k+2)\label{eq:lim-A-4k+2}.
\end{align}
Combining \eqn{lim-8k+4}, \eqn{lim-A-4k+2} and \eqn{zeta-2k}, we obtain
\begin{align*}
A_{4k+2}=\frac{(4k+2)!}{(1-2^{4k+2})\pi^{4k+2}B_{4k+2}}
\end{align*}
and this yields the desired conclusion.
\end{proof}

\begin{cor}\label{cor:theta-8k+4-q} For non-negative integers $k$, we find
\begin{align*}
\theta^{8k+4}_2(\tau)=\frac{-8(2k+1)}{(1-2^{4k+2})B_{4k+2}}\sum_{n=0}^{\infty}\frac{(2n+1)^{4k+1}q^{2n+1}}{1-q^{4n+2}}+T_{4k+2}(\tau).
\end{align*}
\end{cor}

\begin{proof}
From \eqn{wp-pi-2-tau} and \eqn{wp-pi-2-q}, we have
\begin{align}
&(4k+1)!\frac{2^{4k+2}}{\pi^{4k+2}}\sum_{m,n=-\infty}^{\infty}\frac{(-1)^m\chi_2(n)}{(m+n\tau)^{4k+2}}\nonumber\\
=&\wp^{(4k)}(\frac{\pi\tau}{2}|\tau)-\wp^{(4k)}(\frac{\pi+\pi\tau}{2}|\tau)\nonumber\\
=&-2^{4k+4}\sum_{n=0}^{\infty}\frac{(2n+1)^{4k+1}q^{2n+1}}{1-q^{4n+2}}\label{eq:sum-4k+2}.
\end{align}
Combining \eqn{wp-pi-k=0}, \eqn{sum-4k+2} and  \eqn{theta2-8k+4}, we obtain the desired identity.
\end{proof}

We remark that it is easy to see that $\theta^{8k+4}_2(\tau)$ and $M^*_{4k+2}(\tau;\chi_2)$
are modular of weight $4k+2$ with respect to the arithmetic group $\Gamma(2)$ which has fundamental domain $\Omega(2)$ with cusps at 0,1 and $\infty$;
\begin{align*}
\Omega(2)=\{\tau:0\leq x\leq 2, |\tau-1/2|\geq1/2~ and~|\tau-3/2|\geq1/2\}
\end{align*}
is a fundamental domain of $\Gamma(2)$.\\
\section{representations of $\theta^{8k+2}_2(\tau)$ and $\theta^{8k-2}_2(\tau)$}\label{sec:theta2-2k}
Define
\begin{align*}
L(k):=\sum_{n=0}^{\infty}\frac{(-1)^n}{(2n+1)^k}.
\end{align*}
We recall  \cite[p.42  Eq:(16)]{AE}, for all  integers $k\ge 0$,
\begin{align}\label{zeta-2k}
L(2k+1)=(-1)^{k}\pi^{2k+1}\frac{E_{2k}}{2^{2k+2}(2k)!},
\end{align}
 where $E_{2k}$ denotes the Euler's number. We note $(-1)^kE_{2k}>0$.\\

Let
\begin{align*}
M_{2k+1}(\tau;\chi)
=&\sum_{m,n=-\infty}^{\infty}\frac{1}{(4m+1+(2n+1)\tau)^{2k+1}}\\
&+\sum_{m,n=-\infty}^{\infty}\frac{(-1)^{k+1}{ i}}{(4m+2n+2+(2n+1)\tau)^{2k+1}}.
\end{align*}

\begin{lemma} For all integers $k\ge 1$, we have
\begin{align*}
&M_{4k+1}\left(\tau+1;\chi\right)={ i}M_{4k+1}(\tau;\chi),\\
&M_{4k+1}\left(\frac{\tau}{-2\tau+1};\chi\right)=(-2\tau+1)^{4k+1}M_{4k+1}(\tau;\chi),\\
&M_{4k-1}\left(\tau+1;\chi\right)=-{i}M_{4k-1}(\tau;\chi),\\
&M_{4k-1}\left(\frac{\tau}{-2\tau+1};\chi\right)=(-2\tau+1)^{4k-1}M_{4k-1}(\tau;\chi).
\end{align*}
\end{lemma}
\begin{proof}

Let
\begin{align*}
A(\tau)=&\sum_{m,n=-\infty}^{\infty}\frac{1}{(4m+1+(2n+1)\tau)^{4k+1}},\\
B(\tau)=&\sum_{m,n=-\infty}^{\infty}\frac{1}{(4m+2n+2+(2n+1)\tau)^{4k+1}}.
\end{align*}

It is easy to check that
\begin{align*}
A(\tau+1)=B(\tau)\qquad\text{and}\qquad B(\tau+1)=-A(\tau).
\end{align*}
That is, $M_{4k+1}(\tau+1;\chi)=A(\tau+1)+{ i}B(\tau+1)=B(\tau)+{ i} A(\tau)={ i}M_{4k+1}(\tau;\chi)$.

Next,
\begin{align*}
A\left(\frac{\tau}{-2\tau+1}\right)
=&(-2\tau+1)^{4k+1}\sum_{m,n=-\infty}^{\infty}\frac{1}{(4m+1+(2n-8m-1)\tau)^{4k+1}}\\
=&(-2\tau+1)^{4k+1}\sum_{m,n=-\infty}^{\infty}\frac{1}{(4m+1+(2(n-4m-1)+1)\tau)^{4k+1}}\\
=&(-2\tau+1)^{4k+1}\sum_{m,n=-\infty}^{\infty}\frac{1}{(4m+1+(2n+1)\tau)^{4k+1}}\\
=&(-2\tau+1)^{4k+1}A(\tau),
\end{align*}
and
\begin{align*}
B\left(\frac{\tau}{-2\tau+1}\right)
=&(-2\tau+1)^{4k+1}\sum_{m,n=-\infty}^{\infty}\frac{1}{(4m+2n+2+(-2n-8m-3)\tau)^{4k+1}}\\
=&(-2\tau+1)^{4k+1}\sum_{m,n=-\infty}^{\infty}\frac{1}{(4m+2n+2+(2(-n-4m-2)+1)\tau)^{4k+1}}\\
=&(-2\tau+1)^{4k+1}\sum_{m,n=-\infty}^{\infty}\frac{1}{(4m+2n+2+(2n+1)\tau)^{4k+1}}\\
=&(-2\tau+1)^{4k+1}B(\tau).
\end{align*}
That is,
\begin{align*}
 &M_{4k+1}\left(\frac{\tau}{-2\tau+1};\chi\right)\\
=&(-2\tau+1)^{4k+1}A(\tau)-{ i}(-2\tau+1)^{4k+1}B(\tau)\\
=&(-2\tau+1)^{4k+1}M_{4k+1}(\tau;\chi).
\end{align*}
The proofs of the remaining identities are identical, we omit them.
\end{proof}

\begin{theorem} For all integers $k\ge 1$, we obtain
\begin{align}
\theta^{8k+2}_2(\tau)
=&\frac{2(4k)!}{\pi^{4k+1}E_{4k}}M_{4k+1}(\tau;\chi)+T_{4k+1}(\tau),\label{eq:theta2-8k+2}\\
\theta^{8k-2}_2(\tau)
=&\frac{2(4k-2)!}{\pi^{4k-1}E_{4k-2}}M_{4k-1}(\tau;\chi)+ T_{4k-1}(\tau),\label{eq:theta2-8k-2}
\end{align}
where $T_{4k+1}(\tau)\in \mathbf S_{4k+1}(\Gamma_0(2);\psi)$ and $T_{4k-1}(\tau)\in \mathbf S_{4k-1}(\Gamma_0(2);\psi^{-1})$.
\end{theorem}

\begin{proof}
Recalling the facts:
\begin{align*}
\theta^{8k+2}_2(\tau+1)
={ i}\theta^{8k+2}_2(\tau)\quad\text{and}\quad
\theta_2^{8k+2}\left(\frac{\tau}{-2\tau+1}\right)
=(-2\tau+1)^{4k+1} \theta_2^{8k+2}\left(\tau\right),
\end{align*}
we will choose an appropriate constant $A_{4k+1}$, such that as $\tau\rightarrow0$,
\begin{align}\label{eq:A-4k+1}
\lim_{\tau\rightarrow0}\tau^{4k+1}\left(\theta^{8k+2}_2(\tau)-A_{4k+1}M_{4k+1}(\tau;\chi)\right)=0.
\end{align}
Then
\begin{align}\label{eq:defn-A-4k+1}
\theta^{8k+2}_2(\tau)=A_{4k+1}M_{4k+1}(\tau;\chi)+T_{4k+1}(\tau);
\end{align}
where $T_{4k+1}(\tau)\in \mathbf S_{4k+1}(\Gamma_0(2);\psi)$.\\

We now compute the values of $A_{4k+1}$.\\

From \eqn{lim-theta2-i} and \eqn{A-4k+1},
\begin{align*}
\frac{{i}}{2^{4k+1}}=\lim_{\tau\rightarrow0}\tau^{4k+1}\theta^{8k+2}_2(2\tau)=A_{4k+1}\lim_{\tau\rightarrow 0}\tau^{4k+1}M_{4k+1}(2\tau;\chi).
\end{align*}
Then,
\begin{align*}
 &\lim_{\tau\rightarrow0}\tau^{4k+1}M_{4k+1}(2\tau;\chi)\\
=&\lim_{\tau\rightarrow0}\tau^{4k+1}\sum_{m,n=-\infty}^{\infty}\frac{1}{\left(4m+1+(4n+2)\tau\right)^{4k+1}}\\
 &- i \lim_{\tau\rightarrow0}\tau^{4k+1}\sum_{m,n=-\infty}^{\infty}\frac{1}{(4m+2n+2+(4n+2)\tau)^{4k+1}}.
\end{align*}
The limit of the first sum is $0$, to evaluate the second sum, we re-write it as
\begin{align*}
&\tau^{4k+1}\sum_{m,n=-\infty}^{\infty}\frac{1}{(4m+2n+2+(4n+2)\tau)^{4k+1}}\\
=&\tau^{4k+1}\sum_{\substack{m,n=-\infty\\4m+2n+2=0}}^{\infty}\frac{1}{(4m+2n+2+(4n+2)\tau)^{4k+1}}\\
&+\tau^{4k+1}\sum_{\substack{m,n=-\infty\\4m+2n+2\neq0}}^{\infty}\frac{1}{(4m+2n+2+(4n+2)\tau)^{4k+1}}
\end{align*}
and we note, as $\tau\rightarrow0$, the second sum goes to 0 and the first sum becomes
\begin{align*}
 &\tau^{4k+1}\sum_{\substack{m,n=-\infty\\4m+2n+2=0}}^{\infty}\frac{1}{(4m+2n+2+(4n+2)\tau)^{4k+1}}\\
=&-\sum_{m=-\infty}^{\infty}\frac{1}{(4m+1)^{4k+1}}=-\sum_{m=0}^{\infty}\frac{(-1)^m}{(2m+1)^{4k+1}}.
\end{align*}
Letting $\tau\rightarrow0$, we obtain
\begin{align}\label{eq:L-4k+1}
\frac{1}{A_{4k+1}}=\sum_{m=0}^{\infty}\frac{(-1)^m}{(m+1/2)^{4k+1}}.
\end{align}
Substituting \eqn{L-4k+1} into  the equation \eqn{defn-A-4k+1}, we complete the proof of the \eqn{theta2-8k+2}
and \eqn{theta2-8k-2} is identical, we omit it.
\end{proof}

\begin{cor}\label{cor:theta-8kpm2-q} For all positive integers $k$, we have
\begin{align}
\theta^{8k+2}_2(\tau)
=&\frac{-8}{E_{4k}}q^{1/2}\sum_{n=0}^{\infty}\sigma_{4k,\chi}(4n+1)q^{2n}+T_{4k+1}(\tau)\label{eq:theta-8k+2-q}\\
=&\frac{4}{E_{4k}}q^{1/2}\sum_{n=0}^{\infty}(2n+1)^{4k}\left(\frac{(-1)^nq^{n}}{1-q^{2n+1}}+\frac{q^{n}}{1+q^{2n+1}}\right)+T_{4k+1}(\tau),\nonumber\\
\theta^{8k-2}_2(\tau)
=&\frac{8}{E_{4k-2}}q^{1/2}\sum_{n=0}^{\infty}\sigma_{4k-2,\chi}(4n+3)q^{2n+1}+T_{4k-1}(\tau)\label{eq:theta-8k-2-q}\\
=&\frac{-4}{E_{4k-2}}q^{1/2}\sum_{n=0}^{\infty}(2n+1)^{4k}\left(\frac{(-1)^nq^{n}}{1-q^{2n+1}}-\frac{q^{n}}{1+q^{2n+1}}\right)+T_{4k-1}(\tau).\nonumber
\end{align}
\end{cor}

\begin{proof}
From \eqn{wp-2k-1-(-i)-sigma}, we find
\begin{align}
 &(4k)!\left(\frac 4\pi\right)^{4k+1}M_{4k+1}(2\tau;\chi)\label{eq:sum-4k+1}\\
=&\wp^{(4k-1)}(\frac{\pi+2\pi\tau}{4}|\tau)- i\wp^{(4k-1)}(\frac{\pi+\pi\tau}{2}|\tau+1/2)\nonumber\\
=&-2^{4k+3}\sum_{n=0}^{\infty}\sigma_{4k,\chi}(4n+1)q^{4n+1}\nonumber\\
=&2^{4k+2}\sum_{n=0}^{\infty}(2n+1)^{4k}\left(\frac{(-1)^nq^{2n+1}}{1-q^{4n+2}}+\frac{q^{2n+1}}{1+q^{4n+2}}\right)\nonumber.
\end{align}
Substituting \eqn{sum-4k+1} into \eqn{theta2-8k+2}, we get the desired identity \eqn{theta-8k+2-q} and so is \eqn{theta-8k-2-q}.
\end{proof}

\section{A Recurrence relation for the Weierstrass elliptic function }\label{sec:recurrence}

The main goal of this section is to determine the cusp forms appeared in the previous section using an algorithm based on a recurrence relation for $\wp^{(2n)}(z|\tau)$. \\

It follows, from (F), that
\begin{align*}
\wp^{\prime\prime}=6\wp^2-\frac{g_2}{2};
\end{align*}
and from the Leibniz's rule for differentiation, we have, for all positive integers $n$,
\begin{align}\label{eq:wp-recurrence}
\wp^{(n+2)}=6\sum_{k=0}^{n}\binom nk \wp^{(n-k)}\wp^{(k)}.
\end{align}
We conclude from \eqn{theta4-prime} that
\begin{align*}
\wp^{(2n-1)}\left(\frac{\pi\tau}{2}|\tau\right)=0,
\end{align*}
for all positive integers $n$.

This, together with \eqn{wp-recurrence}, gives the following recurrence relation.
\begin{lemma}\label{lem:wp-2n+2} For all positive integers $n$, we have
\begin{align*}
&\wp^{(2n+2)}\left(\frac{\pi\tau}{2}|\tau\right)
=6\sum_{k=0}^{n}\binom{2n}{2k}\wp^{(2n-2k)}\left(\frac{\pi\tau}{2}|\tau\right)\wp^{(2k)}\left(\frac{\pi\tau}{2}|\tau\right).
\end{align*}
\end{lemma}

Thus,
\begin{align}
\mylabel{eq:wp-(4)-pitau}\wp^{(4)}\left(\frac{\pi\tau}{2}|\tau\right)
&=12\wp\left(\frac{\pi\tau}{2}|\tau\right)\wp''\left(\frac{\pi\tau}{2}|\tau\right),\\
\mylabel{eq:wp-(6)-pitau}\wp^{(6)}\left(\frac{\pi\tau}{2}|\tau\right)
&=12\left(\wp\left(\frac{\pi\tau}{2}|\tau\right)\wp^{(4)}\left(\frac{\pi\tau}{2}|\tau\right)+3(\wp''\left(\frac{\pi\tau}{2}|\tau\right))^2\right)\\
&=36\left(4\wp''\left(\frac{\pi\tau}{2}|\tau\right)\wp^2\left(\frac{\pi\tau}{2}|\tau\right)+(\wp''\left(\frac{\pi\tau}{2}|\tau\right))^2\right),\nonumber\\
\mylabel{eq:wp-(8)-pitau}\wp^{(8)}\left(\frac{\pi\tau}{2}|\tau\right)
&=12\left(\wp\left(\frac{\pi\tau}{2}|\tau\right)\wp^{(6)}\left(\frac{\pi\tau}{2}|\tau\right)
+15\wp''\left(\frac{\pi\tau}{2}|\tau\right)\wp^{(4)}\left(\frac{\pi\tau}{2}|\tau\right)\right)\\
&=12^3\left(\wp''\left(\frac{\pi\tau}{2}|\tau\right)\wp^3\left(\frac{\pi\tau}{2}|\tau\right)
+\frac32(\wp''\left(\frac{\pi\tau}{2}|\tau\right))^2\wp\left(\frac{\pi\tau}{2}|\tau\right)\right),\nonumber
\end{align}
and
\begin{align}
\mylabel{eq:wp-(10)-pitau}\wp^{(10)}\left(\frac{\pi\tau}{2}|\tau\right)
&=12\left(\wp\left(\frac{\pi\tau}{2}|\tau\right)\wp^{(8)}\left(\frac{\pi\tau}{2}|\tau\right)
+28\wp''\left(\frac{\pi\tau}{2}|\tau\right)\wp^{(6)}\left(\frac{\pi\tau}{2}|\tau\right)+35(\wp^{(4)}\left(\frac{\pi\tau}{2}|\tau\right))^2\right)\\
&=12^3\left(12\wp''\left(\frac{\pi\tau}{2}|\tau\right)\wp^4\left(\frac{\pi\tau}{2}|\tau\right)
+81(\wp\left(\frac{\pi\tau}{2}|\tau\right)\wp''\left(\frac{\pi\tau}{2}|\tau\right))^2+7(\wp''\left(\frac{\pi\tau}{2}|\tau\right))^3\right).\nonumber
\end{align}

It is clear that we can express  $\wp^{(2k)}\left(\frac{\pi\tau}{2}|\tau\right)$ as polynomial of $\wp\left(\frac{\pi\tau}{2}|\tau\right)$ and $\wp''\left(\frac{\pi\tau}{2}|\tau\right)$:
\begin{align*}
\wp^{(2k)}\left(\frac{\pi\tau}{2}|\tau\right)=P_{2k}\left(\wp\left(\frac{\pi\tau}{2}|\tau\right),\wp''\left(\frac{\pi\tau}{2}|\tau\right)\right);
\end{align*}
where $P_{2k}(x,y)$ is polynomial in $x$ and $y$. \\

Taking $z=0$ in \eqn{wp-(2k)-pitau}, we have
\begin{align*}
\wp^{(2k)}\left(\frac{\pi\tau}{2}|\tau\right)
=(-1)^{k+1}2^{2k+3}\sum_{n=1}^{\infty}\frac{n^{2k+1} q^n}{1-q^{2n}}.
\end{align*}\\

Recall, from (B) and (F),
\begin{align}
&\wp\left(\frac{\pi\tau}{2}|\tau\right)
=-\frac13\left(\theta^4_2(\tau)+\theta^4_3(\tau)\right),\nonumber\\
&\wp^{\prime\prime}\left(\frac{\pi\tau}{2}|\tau\right)
=2\theta^4_2(\tau)\theta^4_3(\tau)=\frac18\theta_2^8(\tau/2)\mylabel{eq:wp-(2)-pitau},
\end{align}
we derive
\begin{align}\label{eq:P_2k}
(-1)^{k+1}2^{2k+3}\sum_{n=1}^{\infty}\frac{n^{2k+1} q^n}{1-q^{2n}}
=P_{2k}\left(-\frac13(\theta^4_2(\tau)+\theta^4_3(\tau)),2\theta^4_2(\tau)\theta^4_3(\tau)\right).
\end{align}
For $k=1$, replacing $\tau$ by $2\tau$, we obtain from \eqn{wp-(2)-pitau},
\begin{align*}
\theta^8_2(\tau)=2^8\sum_{n=1}^{\infty}\frac{n^3q^{2n}}{1-q^{4n}}.
\end{align*}

We now establish:
\begin{align*}
691\theta^{24}_2(\tau)=2^{16}\sum_{n=1}^{\infty}\frac{n^{11}q^{n}}{1-q^{2n}}-2^{16}q(q;q)^{24}_\infty
-259\times2^{19}q^2(q^2;q^2)^{24}_\infty.
\end{align*}

\begin{proof}

From \eqn{wp-(10)-pitau} and appealing to the elementary facts of theta functions:
\begin{align*}
\theta_2(\tau)\theta_3(\tau)\theta_4(\tau)
=&2q^{1/4}(q^2;q^2)^3_\infty,\\
\theta_2(\tau)\theta_3(\tau)\theta_4^4(\tau)
=&2q^{1/4}(q;q)^6_\infty,\\
(\theta^4_2(\tau)+\theta^4_3(\tau))^2
=&(\theta^4_2(\tau)-\theta^4_3(\tau))^2+4 \theta_2^4(\tau)\theta_3^4(\tau)
=\theta_4^8(\tau)+4 \theta_2^4(\tau)\theta_3^4(\tau),\\
(\theta^4_2(\tau)+\theta^4_3(\tau))^4
=&\theta_4^{16}(\tau)+8 \theta_2^4(\tau)\theta_3^4(\tau)\theta_4^8(\tau) +16 \theta_2^8(\tau)\theta_3^8(\tau),
\end{align*}
we have
\begin{align*}
 &2^{13}\sum_{n=1}^{\infty}\frac{n^{11} q^n}{1-q^{2n}}\\
=&\wp^{(10)}\left(\frac{\pi\tau}{2}|\tau\right)\\
= &12^3\big\{12 (2\theta^4_2(\tau)\theta^4_3(\tau))\times\frac{1}{3^4}(\theta^4_2(\tau)+\theta^4_3(\tau))^4\\
&+81 (2\theta^4_2(\tau)\theta^4_3(\tau))^2\times\frac{1}{3^2}(\theta^4_2(\tau)+\theta^4_3(\tau))^2+7(2\theta^4_2(\tau)\theta^4_3(\tau))^3\big\}\\
=&2^8\big\{1382(\theta_2(\tau)\theta_3(\tau))^{12}+259(\theta_2(\tau)\theta_3(\tau)\theta_4(\tau))^8+2(\theta_2(\tau)\theta_3(\tau)\theta_4^4(\tau))^{4}\big\}\\
=&691\times2^{-3}\theta^{24}_2(\tau/2)+2^{16} \times 259 q^2(q^2;q^2)^{24}_\infty+2^{13} q (q;q)^{24}_\infty.
\end{align*}
Replacing $\tau$ by $2\tau$ and rearranging the terms, we obtain the stated identity.
\end{proof}

Next, we prove
\begin{align*}
\theta^{12}_2(\tau)
=16\sum_{n=0}^{\infty}\frac{(2n+1)^5q^{2n+1}}{1-q^{4n+2}}-16q(q^2;q^2)^{12}_\infty.
\end{align*}
\begin{proof}
From \eqn{wp-(4)-pitau},
\begin{align*}
-2^7\sum_{n=1}^{\infty}\frac{n^5q^n}{1-q^{2n}}
=\wp^{(4)}\left(\frac{\pi\tau}{2}|\tau\right)\\
=-8\left(\theta^4_2(\tau)+\theta^4_3(\tau)\right)\left(\theta^4_2(\tau)\theta^4_3(\tau)\right).
\end{align*}
Then
\begin{align*}
16\sum_{n=1}^{\infty}\frac{n^5q^n}{1-q^{2n}}=\theta^4_2(\tau)\theta^4_3(\tau)\left(\theta^4_3(\tau)+\theta^4_2(\tau)\right)
\end{align*}
and replacing $\tau$ by $\tau+1$, we obtain
\begin{align*}
16\sum_{n=1}^{\infty}\frac{(-1)^n n^5 q^n}{1-q^{2n}}=-\theta^4_2(\tau)\theta^4_4(\tau)(\theta^4_4(\tau)-\theta^4_2(\tau)).
\end{align*}
Subtracting the above identities,
\begin{align*}
 &32\sum_{n=1}^{\infty}\frac{(2n+1)^5 q^{2n+1}}{1-q^{4n+2}}\\
=&\theta^4_2(\tau)(\theta^8_3(\tau)+\theta^8_4(\tau))+\theta^8_2(\tau)(\theta^4_3(\tau)-\theta^4_4(\tau))\\
=&\theta^4_2(\tau)(\theta^8_3(\tau)+\theta^8_4(\tau))+\theta^{12}_2(\tau)\\
=&\theta^4_2(\tau)(\theta^4_3(\tau)-\theta^4_4(\tau))^2+2\theta^4_2(\tau)\theta^4_3(\tau)\theta^4_4(\tau)+\theta^{12}_2(\tau)\\
=&2\theta^{12}_2(\tau)+2\theta^4_2(\tau)\theta^4_3(\tau)\theta^4_4(\tau)\\
=&2\theta^{12}_2(\tau)+32q(q^2;q^2)^{12}_{\infty}.
\end{align*}
This establishes the desired identity.\\
\end{proof}

We now consider the cases of $\theta_2^6(\tau)$, $\theta_2^{10}(\tau)$ and $\theta_2^{14}(\tau)$. Their formulas are derived from Theorem \thm{wp-2k-1} with $k=1$, $k=2$ and $k=3$.\\

We begin with  $\theta_2^6(\tau)$.\\

Recall, from (E),
\begin{align*}
\wp^{\prime}\left(\frac{\pi+2\pi\tau}{4}|\tau\right)
=4\theta^2_2(2\tau)\theta^4_4(2\tau).
\end{align*}
Taking $k=1$ and $z=0$ in \eqn{wp-(2k-1)-pitau}, we have
\begin{align*}
\wp^{\prime}\left(\frac{\pi+2\pi\tau}{4}|\tau\right)
=16\sum_{n=1}^{\infty}\frac{n^2 q^n}{1-q^{2n}}\sin\frac{n\pi}{2}
\end{align*}
or, after replacing $\tau$ by $\tau/2$,
\begin{align*}
\theta_2^2(\tau)\theta_4^4(\tau)=4q^{1/2}\sum_{n=0}^{\infty}(-1)^n\frac{(2n+1)^2 q^n}{1-q^{2n+1}}.
\end{align*}
Replacing $q$ by $-q$ (or equivalently $\tau$ by $\tau+1$), we obtain
\begin{align*}
\theta_2^2(\tau)\theta_3^4(\tau)=4q^{1/2}\sum_{n=0}^{\infty}\frac{(2n+1)^2 q^n}{1+q^{2n+1}}.
\end{align*}
Thus
\begin{align*}
\theta_2^6(\tau)
=\theta_2^2(\tau)(\theta_3^4(\tau)-\theta_4^4(\tau))
=4q^{1/2}\sum_{n=0}^{\infty}(2n+1)^2\left(\frac{q^n}{1+q^{2n+1}}-\frac{(-1)^nq^n}{1-q^{2n+1}}\right).
\end{align*}
Notably,
\begin{align*}
\wp^{\prime}\left(\frac{\pi+\pi\tau}{4}|\frac{\tau}{2}\right)+i\wp^{\prime}\left(\frac{2\pi+\pi\tau}{4}|\frac{\tau+1}{4}\right)
=-4\theta_2^6(\tau).
\end{align*}

Next we consider  $\theta_2^{10}(\tau)$ and  $\theta_2^{14}(\tau)$.\\

From (E), we have

\begin{align*}
&\wp\left(\frac{\pi+2\pi\tau}{4}|\tau\right)
=-\frac13(\theta^4_3(2\tau)-5\theta^4_2(2\tau)),\\
&\wp^{\prime}\left(\frac{\pi+2\pi\tau}{4}|\tau\right)
=4\theta^2_2(2\tau)\theta^4_4(2\tau).
\end{align*}
And the recurrence relation \eqn{wp-recurrence} with $n=1$ and $n=3$, we have
\begin{align*}
&\wp^{(3)}\left(\frac{\pi+2\pi\tau}{4}|\tau\right)\\
=&12\wp\left(\frac{\pi+2\pi\tau}{4}|\tau\right)\wp^{\prime}\left(\frac{\pi+2\pi\tau}{4}|\tau\right)\\
=&12\left(-\frac13(\theta^4_3(2\tau)-5\theta^4_2(2\tau))\right)\left(4\theta^2_2(2\tau)\theta^4_4(2\tau)\right)\\
=&16\theta^2_2(2\tau)\theta^4_4(2\tau)(5\theta^4_2(2\tau)-\theta^4_3(2\tau))
\end{align*}
and
\begin{align*}
&\wp^{(5)}\left(\frac{\pi+2\pi\tau}{4}|\tau\right)\\
=&6\sum_{k=0}^{3}\binom 3k \wp^{(3-k)}\left(\frac{\pi+2\pi\tau}{4}|\tau\right)\wp^{(k)}\left(\frac{\pi+2\pi\tau}{4}|\tau\right)\\
=&12\left(\wp\left(\frac{\pi+2\pi\tau}{4}|\tau\right)\wp^{(3)}\left(\frac{\pi+2\pi\tau}{4}|\tau\right)
+3\wp^{\prime}\left(\frac{\pi+2\pi\tau}{4}|\tau\right)\wp^{\prime\prime}\left(\frac{\pi+2\pi\tau}{4}|\tau\right)\right)\\
=&-2304\theta^6_2(2\tau)\theta^8_4(2\tau)+64\theta^2_2(2\tau)\theta^4_4(2\tau)\left(5\theta^4_2(2\tau)-\theta^4_3(2\tau)\right)^2.
\end{align*}
Replacing $\tau$ by $\tau+1/2$, we have
\begin{align*}
&\wp^{(3)}\left(\frac{\pi+\pi\tau}{2}|\tau\right)
=16{ i}\theta^2_2(2\tau)\theta^4_3(2\tau)(-5\theta^4_2(2\tau)-\theta^4_4(2\tau)),\\
&\wp^{(5)}\left(\frac{\pi+\pi\tau}{2}|\tau\right)
=-2304{ i^3}\theta^6_2(2\tau)\theta^8_3(2\tau)+64{ i}\theta^2_2(2\tau)\theta^4_3(2\tau)\left(-5\theta^4_2(2\tau)-\theta^4_4(2\tau)\right)^2.
\end{align*}
Then
\begin{align*}
&\wp^{(3)}\left(\frac{\pi+2\pi\tau}{4}|\tau\right)-{ i}\wp^{(3)}\left(\frac{\pi+\pi\tau}{2}|\frac{2\tau+1}{2}\right)\\
=&16\theta^2_2(2\tau)\theta^4_4(2\tau)(5\theta^4_2(2\tau)-\theta^4_3(2\tau))
-16{ i}^2\theta^2_2(2\tau)\theta^4_3(2\tau)(-5\theta^4_2(2\tau)-\theta^4_4(2\tau))\\
=&-32\theta^2_2(2\tau)\theta^4_3(2\tau)\theta^4_4(2\tau)+80\theta^6_2(2\tau)\theta^4_4(2\tau)-80\theta^6_2(2\tau)\theta^4_3(2\tau)\\
=&-32\theta^2_2(2\tau)\theta^4_3(2\tau)\theta^4_4(2\tau)-80\theta^{10}_2(2\tau)
\end{align*}
and
\begin{align*}
&\wp^{(5)}\left(\frac{\pi+2\pi\tau}{4}|\tau\right)+{ i}\wp^{(5)}\left(\frac{\pi+\pi\tau}{2}|\frac{2\tau+1}{2}\right)\\
=&-2304\theta^6_2(2\tau)\theta^8_4(2\tau)+64\theta^2_2(2\tau)\theta^4_4(2\tau)\left(5\theta^4_2(2\tau)-\theta^4_3(2\tau)\right)^2\\
&-2304{ i^4}\theta^6_2(2\tau)\theta^8_3(2\tau)+64{ i^2}\theta^2_2(2\tau)\theta^4_3(2\tau)\left(-5\theta^4_2(2\tau)-\theta^4_4(2\tau)\right)^2\\
=&-2^6\times61\theta^{14}_2(2\tau)-2^6\times91\theta^6_2(2\tau)\theta^4_3(2\tau)\theta^4_4(2\tau).
\end{align*}

Hence,
\begin{align*}
80\theta^{10}_2(2\tau)=&-\wp^{(3)}\left(\frac{\pi+2\pi\tau}{4}|\tau\right)+{ i}\wp^{(3)}\left(\frac{\pi+\pi\tau}{2}|\frac{2\tau+1}{2}\right)\\
&-32\theta^2_2(2\tau)\theta^4_3(2\tau)\theta^4_4(2\tau),\\
2^6\times61\theta^{14}_2(2\tau)
=&-\wp^{(5)}\left(\frac{\pi+2\pi\tau}{4}|\tau\right)-{ i}\wp^{(5)}\left(\frac{\pi+\pi\tau}{2}|\frac{2\tau+1}{2}\right)\\
&-2^6\times91\theta^6_2(2\tau)\theta^4_3(2\tau)\theta^4_4(2\tau).
\end{align*}

The desired identities follow from \eqn{wp-2k-1-(-i)-q}, \eqn{wp-2k-1-(i)-q} and Jacobi's product identities for theta functions.\\

To complete the proofs of the identities listed in Section \sect{intr}, additional identities derived from the recurrence relation of
$\wp(z|\tau)$ are provided in Section \sect{comment}.

\section{Comments}\label{sec:comment}

We start with a set of elementary lemmas which might be of independent interest in themselves.

\begin{lemma}

Suppose
\begin{align*}
\frac{p_n(x)}{(1-x)^n}=\sum_{k=0}^{\infty}k^{n-1}x^k.
\end{align*}
Then,
\begin{align}
p_1(x)&=1,\nonumber\\
p_{n+1}(x)&=nxp_n(x)+x(1-x)p'_n(x)\mylabel{eq:p-n+1}.
\end{align}
\end{lemma}
Thus, $p_2(x)=x$, $p_3(x)=x^2+x$ and $p_4(x)=x^3+4x^2+x$.\\

\begin{proof}
We note that
\begin{align*}
x\frac{d}{dx}\frac{p_n(x)}{(1-x)^n}=\sum_{k=0}^{\infty}k^{n}x^k=\frac{p_{n+1}(x)}{(1-x)^{n+1}}
\end{align*}
and from which we derive the recurrence relation.
\end{proof}

Using the recurrence relation \eqn{p-n+1}, we claim that the coefficients of $p_{n}(x)$ are palindromic.\\

\begin{lemma}
Suppose
\begin{align*}
p_n(x)=a_{n-1,n} x^{n-1}+a_{n-2, n}x^{n-2}+... +a_{2,n}x^2+a_{1,n}x.
\end{align*}
Then $a_{n-1,n}=a_{1,n}=1$ and
\begin{align}\label{eq:a-j,n}
a_{j,n}=a_{n-j,n},
\end{align}
for $ j=1,2,..,n-1$.\\

\end{lemma}
\begin{proof}

We will prove by induction.\\

Suppose
\begin{align*}
p_n(x)= x^{n-1}+a_{n-2, n}x^{n-2}+... +a_{2,n}x^2+x.
\end{align*}
Assume the coefficients of $p_n$ are palindromic.

Then $a_{n-1,n}=a_{1,n}=1$ and
\begin{align}\label{a-j,n}
a_{j,n}=a_{n-j,n},
\end{align}
for $ j=1,2,..,n-1$.\\

Let
\begin{align*}
p_{n+1}(x)= x^{n}+a_{n-1, n+1}x^{n-1}+... +a_{2,n+1}x^2+x.
\end{align*}

From \eqn{p-n+1}, for $m=1,2,..., n$,
\begin{align*}
a_{m, n+1}=(n+1-m)a_{m-1,n}+ma_{m,n}.
\end{align*}
Then, from \eqn{a-j,n},
\begin{align*}
a_{n+1-m, n+1}=ma_{a_{n-m,n}}+(n+1-m)a_{n+1-m,n}=a_{m, n+1}.
\end{align*}
This establishes the claim.\\

\end{proof}

In fact, we have
\begin{lemma} Let
\begin{align*}
p_n(x)= x^{n-1}+a_{n-2, n}x^{n-2}+... +a_{2,n}x^2+x.
\end{align*}
Then, for $1\le m\le n-1$,
\begin{align*}
a_{m,n}= \sum_{j=0}^m (-1)^j\binom{n}{j}(m-j)^{n-1}.
\end{align*}
\end{lemma}
\begin{proof}
\begin{align*}
p_n(x)
=&(1-x)^n\sum_{k=0}^\infty k^{n-1}x^k\\
=&\sum_{j=0}^n (-1)^j\binom{n}{j}x^j \sum_{k=0}^\infty k^{n-1}x^k\\
=&\sum_{m=1}^{n-1} x^m  \sum_{j=0}^m (-1)^j\binom{n}{j}(m-j)^{n-1}\\
&+\sum_{m=n}^\infty x^m  \sum_{j=0}^n (-1)^j\binom{n}{j}(m-j)^{n-1}\\
=&\sum_{m=1}^{n-1} x^m  \sum_{j=0}^m (-1)^j\binom{n}{j}(m-j)^{n-1}.
\end{align*}
This gives the desired result.\\
\end{proof}

Moreover,
as bonuses, we also derive additional identities:
\begin{align*}
 \sum_{k=0}^{n-1} (-1)^j\binom{n}{j}(n-1-j)^{n-1}=1
 \end{align*}
and for all $m\ge n$,
\begin{align*}
 \sum_{j=0}^n (-1)^j\binom{n}{j}(m-j)^{n-1}=0.
\end{align*}

\begin{lemma} For $|q|<1$, we have
\begin{align*}
\sum_{n=1}^{\infty}\frac{n^{k-1} q^n}{1-q^{2n}}=\sum_{n=0}^{\infty}\frac{p_k(q^{2n+1})}{(1-q^{2n+1})^k}.
\end{align*}
\end{lemma}
\begin{proof}

Since
\begin{align*}
\frac{q^n}{1-q^{2n}}=\sum_{j=0}^{\infty}q^{(2j+1)n},
\end{align*}
then
\begin{align*}
\sum_{n=1}^{\infty}\frac{n^{k-1} q^n}{1-q^{2n}}
&=\sum_{n=1}^{\infty}\sum_{j=0}^{\infty}n^{k-1}q^{(2j+1)n}\\
&=\sum_{j=0}^{\infty}\sum_{n=1}^{\infty}n^{k-1}q^{(2j+1)n}\\
&=\sum_{j=0}^{\infty}\frac{p_k(q^{2j+1})}{(1-q^{2j+1})^k}.
\end{align*}
\end{proof}

From Corollary \corol{theta-8k-q},
\begin{cor}\label{cor:theta-8k-p}
 For integer $k\geq1$, we find
\begin{align*}
\theta^{8k}_2(\tau)=\frac{2^{4k+3}k}{(1-2^{4k})B_{4k}}\sum_{n=1}^{\infty}\frac{p_{4k}(q^{4n+2})}{(1-q^{4n+2})^{4k}}+T_{4k}(\tau).
\end{align*}\\
\end{cor}

Similarly, let
\begin{align*}
\frac{P_n(x)}{(1-x)^n}=\sum_{k=0}^{\infty}(2k+1)^{n-1}x^k.
\end{align*}
We note
\begin{align*}
(2x\frac{d}{dx}+1)\frac{P_n(x)}{(1-x)^n}=\sum_{k=0}^{\infty}(2k+1)^{n}x^k=\frac{P_{n+1}(x)}{(1-x)^{n+1}}
\end{align*}
and from which we derive
\begin{lemma}
Suppose
\begin{align*}
\frac{P_n(x)}{(1-x)^n}=\sum_{k=0}^{\infty}(2k+1)^{n-1}x^k.
\end{align*}
Then
\begin{align*}
P_1(x)=1\qquad\text{and}\qquad
P_{n+1}(x)=((2n-1)x+1)P_n(x)+2x(1-x)P'_n(x);
\end{align*}
the coefficients of $P_n(x)$ are palindromic
and
\begin{align*}
\sum_{n=1}^{\infty}\frac{(2n+1)^{k-1} q^n}{1-q^{2n}}=\sum_{n=0}^{\infty}\frac{P_k(q^{2n+1})}{(1-q^{2n+1})^k}.
\end{align*}
\end{lemma}

We can re-write  Corollary \corol{theta-8k+4-q} as

\begin{cor} \label{cor:theta-8k+4-p}
 For integer $k\geq0$, we have
\begin{align*}
\theta^{8k+4}_2(\tau)=\frac{-8(2k+1)}{(1-2^{4k+2})B_{4k+2}}\sum_{n=0}^{\infty}\frac{P_{4k+2}(q^{2n+1})}{(1-q^{2n+1})^{4k+2}}+T_{4k+2}(\tau).
\end{align*}
\end{cor}

To give the details of  \eqref{theta-2k-tran}, we need recall Dedekind's transformation formula for $\eta(\tau)$ in \cite[p. 10]{kolberg}.
\begin{lemma}
Let
$\left(\begin{array}{cc}
a&b\\
c&d
\end{array}\right) \in SL(2,\mathbb{Z })$. We have

(1)If $d>0$ and $d$ is odd,
\begin{align}\label{dodd}
\eta\left(\frac{a\tau+b}{c\tau+d}\right)
=\left(\frac{c}{d}\right)e^{(d(b-c)+ac(1-d^2)+3d-3)\pi{ i}/{12}}\sqrt{c\tau+d}\eta(\tau),
\end{align}

(2)If $c>0$ and $c$ is odd,
\begin{align}\label{codd}
\eta\left(\frac{a\tau+b}{c\tau+d}\right)
=\left(\frac{d}{c}\right)e^{(c(a+d)+bd(1-c^2)-3c+3)\pi{ i}/{12}}\sqrt{-{ i}(c\tau+d)}\eta(\tau),
\end{align}
\end{lemma}

\begin{theorem}\label{thm:theta2-trans} Let $\left(\begin{array}{cc}
a&b\\
c&d
\end{array}\right) \in SL(2,\mathbb{Z })$. We have

(1) when $c\equiv2\pmod4$,
\begin{align}
\theta_2\left(\frac{a\tau+b}{c\tau+d}\right)
=\left(\frac dc\right)\sqrt{-{ i}(c\tau+d)}e^{\frac{\pi{ i}}{4}(bd+1)}\theta_2(\tau);
\end{align}
(2) when $c\equiv0\pmod4$,
\begin{align}
\theta_2\left(\frac{a\tau+b}{c\tau+d}\right)
=\left(\frac cd\right)\sqrt{c\tau+d}e^{\frac{\pi{ i}}{4}(bd+d-1)}\theta_2(\tau).
\end{align}
Especially,
\begin{align*}
\theta_2^{2}\left(\frac{a\tau+b}{c\tau+d}\right)=\psi(\sigma)
(c\tau+d)\theta_2^{2}\left(\tau\right),
\end{align*}
where
\begin{align*}
\psi(\sigma)= \begin{cases} -{ i}e^{\frac{\pi{ i}}{2}(bd+1)} \quad & \text{if }  c\equiv2\pmod 4, \\[0.1in]
       e^{\frac{\pi{ i}}{2}(bd+d-1)} \quad & \text{if } c\equiv0\pmod 4.
\end{cases}
\end{align*}
\end{theorem}

\begin{proof}
Consider a transformation $\tau\to(a\tau+b)/(c\tau+d)$
\begin{align*}
\frac{\eta^2(2\tau)}{\eta(\tau)}\to\frac{\eta^2\left(\frac{a\cdot2\tau+2b}{c/2\cdot2\tau+d}\right)}{\eta\left(\frac{a\tau+b}{c\tau+d}\right)}.
\end{align*}

Now when $c\equiv2\pmod4$, and hence, by \eqref{dodd}

\begin{align*}
&\theta_2\left(\frac{a\tau+b}{c\tau+d}\right)\\
=&\frac{\{\left(\frac{d}{c/2}\right)\sqrt{-{ i}(c/2\cdot2\tau+d)}e^{(c/2(a+d)+2bd(1-c^2/4)-3/2c+3)\pi{ i}/{12}}\eta(2\tau)\}^2}
       {\left(\frac{d}{c}\right)\sqrt{-{ i}(c\tau+d)}e^{(c(a+d)+bd(1-c^2)-3c+3)\pi{ i}/{12}}\eta(\tau)}\\
=&2\left(\frac dc \right)\sqrt{-{ i}(c\tau+d)}e^{\frac{\pi{ i}}{4}(bd+1)}\frac{\eta^2(2\tau)}{\eta(\tau)}\\
=&\left(\frac dc \right)\sqrt{-{ i}(c\tau+d)}e^{\frac{\pi{ i}}{4}(bd+1)}\theta_2(\tau),
\end{align*}

Next when $c\equiv0\pmod4$, and hence, by \eqref{codd}

\begin{align*}
&\theta_2\left(\frac{a\tau+b}{c\tau+d}\right)\\
=&\frac{\{\left(\frac{c/2}{d}\right)\sqrt{(c/2\cdot2\tau+d)}e^{(d(2b-c/2)+ac/2(1-d^2)+3d-3)\pi{ i}/{12}}\eta(2\tau)\}^2}
       {\left(\frac{c}{d}\right)\sqrt{c\tau+d}e^{(d(b-c)+ac(1-d^2)+3d-3)\pi{ i}/{12}}\eta(\tau)}\\
=&2\left(\frac cd \right)\sqrt{c\tau+d}e^{\frac{\pi{ i}}{4}(bd+d-1)}\frac{\eta^2(2\tau)}{\eta(\tau)}\\
=&\left(\frac cd \right)\sqrt{c\tau+d}e^{\frac{\pi{ i}}{4}(bd+d-1)}\theta_2(\tau).
\end{align*}

We get the desired identities.
\end{proof}

\section{A List of Identities}\label{sec:list}

For the convenience of the reader, we provide the following set of identities which are derived from the recurrence
relation of $\wp(z|\tau)$ mentioned in Section \sect{recurrence} and are used to establish the identities of $\theta_2^{2n}(\tau)$, for $n= 2,3,4,..,12$, mentioned in Section \sect{intr}.\\

\begin{align*}
&2^4\sum_{n=1}^{\infty}\frac{n^{3}q^n}{1-q^{2n}}
=\theta^4_2(\tau)\theta^4_3(\tau)=2^{-4}\theta^8_2(\tau/2),\\
&2^5\sum_{n=0}^{\infty}\frac{(2n+1)^{3}q^{2n+1}}{1-q^{4n+2}}
=\theta^4_2(\tau)\left(\theta^4_3(\tau)+\theta^4_4(\tau)\right);
\end{align*}

\begin{align*}
&2^4\sum_{n=1}^{\infty}\frac{n^{5}q^n}{1-q^{2n}}
=\theta^4_2(\tau)\theta^4_3(\tau)\left(\theta^4_2(\tau)+\theta^4_3(\tau)\right),\\
&2^5\sum_{n=0}^{\infty}\frac{(2n+1)^{5}q^{2n+1}}{1-q^{4n+2}}
=\theta^4_2(\tau)\left(2\theta^8_2(\tau)+2\theta^4_3(\tau)\theta^4_4(\tau)\right)\\
&\qquad\qquad\qquad\qquad\qquad=2\theta^{12}_2(\tau)+2\theta^4_2(\tau)\theta^4_3(\tau)\theta^4_4(\tau);
\end{align*}

\begin{align*}
&2^5\sum_{n=1}^{\infty}\frac{n^{7}q^n}{1-q^{2n}}
=\theta^4_2(\tau)\theta^4_3(\tau)\left(2\theta^8_4(\tau)+17\theta^4_2(\tau)\theta^4_3(\tau)\right)\\
&\qquad\qquad\qquad=\theta^4_2(\tau)\theta^4_3(\tau)\theta^8_4(\tau)+17\theta^8_2(\tau)\theta^8_3(\tau),\\
&2^{6}\sum_{n=0}^{\infty}\frac{(2n+1)^{7}q^{2n+1}}{1-q^{4n+2}}
=\theta^4_2(\tau)\left(\theta^4_3(\tau)+\theta^4_4(\tau)\right)
 \left(17\theta^8_2(\tau)+2\theta^4_3(\tau)\theta^4_4(\tau)\right);
 \end{align*}

\begin{align*}
&2^4\sum_{n=1}^{\infty}\frac{n^{9}q^n}{1-q^{2n}}
=\theta^4_2(\tau)\theta^4_3(\tau)\left(\theta^4_2(2\tau)+\theta^4_3(\tau)\right)
  \left(\theta^8_2(\tau)+29\theta^4_2(\tau)\theta^4_3(\tau)+\theta^8_3(\tau)\right),\\
&2^5\sum_{n=0}^{\infty}\frac{(2n+1)^{9}q^{2n+1}}{1-q^{4n+2}}
=62\theta^{20}_2(\tau)+154\theta^{12}_2(\tau)\theta^4_3(\tau)\theta^4_4(\tau)
  +2\theta^4_2(\tau)\theta^8_3(\tau)\theta^8_4(\tau);
 \end{align*}

 \begin{align*}
&2^{5}\sum_{n=1}^{\infty}\frac{n^{11}q^n}{1-q^{2n}}
=4\theta^4_2(\tau)\theta^4_3(\tau)\theta^{16}_4(\tau)
 +259\theta^8_2(\tau)\theta^8_3(\tau)\theta^{8}_4(\tau)
 +1382\theta^{12}_2(\tau)\theta^{12}_3(\tau),\\
 &2^6\sum_{n=0}^{\infty}\frac{(2n+1)^{11}q^{2n+1}}{1-q^{4n+2}}
=\theta^4_2(\tau)\left(\theta^4_3(2\tau)+\theta^4_4(\tau)\right)\\
&\qquad\qquad\qquad\qquad\qquad\times \left(1383\theta^{16}_2(\tau)+1131\theta^8_2(\tau)\theta^4_3(\tau)\theta^4_4(\tau)+2\theta^{12}_3(\tau)\theta^4_4(\tau)\right).
\end{align*}

From Theorem \thm{wp-2k-1},

\begin{align*}
&\sum_{n=0}^{\infty}(2n+1)^{2k}\left(\frac{q^{2n+1}}{1+q^{2n+1}}+\frac{(-1)^nq^{2n+1}}{1-q^{4n+2}}\right)
=-2\sum_{n=0}^\infty \sigma_{2k,\chi}(4n+1)q^{4n+1},\\
&\sum_{n=0}^{\infty}(2n+1)^{2k}\left(\frac{q^{2n+1}}{1+q^{2n+1}}-\frac{(-1)^nq^{2n+1}}{1-q^{4n+2}}\right)
=-2\sum_{n=0}^\infty \sigma_{2k,\chi}(4n+3)q^{4n+3}.
\end{align*}\\

Instead of using the Eisenstein series, alternatively, we have
\begin{align*}
&2^3\sum_{n=0}^{\infty}\sigma_{2,\chi}(4n+1)q^{4n+1}
=\theta^2_2(2\tau)(\theta^4_3(2\tau)+\theta^4_4(2\tau)),\\
&2^3\sum_{n=0}^{\infty}\sigma_{2,\chi}(4n+3)q^{4n+3}
=-\theta^6_2(2\tau);
\end{align*}

\begin{align*}
&2^3\sum_{n=0}^{\infty}\sigma_{4,\chi}(4n+1)q^{4n+1}
=5\theta^{10}_2(2\tau)+2\theta^2_2(2\tau)\theta^4_3(2\tau)\theta^4_4(2\tau),\\
&2^3\sum_{n=0}^{\infty}\sigma_{4,\chi}(4n+3)q^{4n+3}
=-5\theta^{6}_2(2\tau)(\theta^4_3(2\tau)+\theta^4_4(2\tau));
\end{align*}

\begin{align*}
&2^3\sum_{n=0}^{\infty}\sigma_{6,\chi}(4n+1)q^{4n+1}
=\theta^2_2(2\tau)(61\theta^8_2(2\tau)+\theta^4_3(2\tau)\theta^4_4(2\tau))(\theta^4_3(2\tau)+\theta^4_4(2\tau)),\\
&2^3\sum_{n=0}^{\infty}\sigma_{6,\chi}(4n+3)q^{4n+3}
=-61\theta^{14}_2(2\tau)-91\theta^6_2(2\tau)\theta^4_3(2\tau)\theta^4_4(2\tau);
\end{align*}

\begin{align*}
&2^{3}\sum_{n=0}^{\infty}\sigma_{8,\chi}(4n+1)q^{4n+1}
=1385\theta^{18}_2(2\tau)+3052\theta^{10}_2(2\tau)\theta^4_3(2\tau)\theta^4_4(2\tau)
+2\theta^2_2(2\tau)\theta^8_3(2\tau)\theta^8_4(2\tau),\\
&2^3\sum_{n=0}^{\infty}\sigma_{8,\chi}(4n+3)q^{4n+3}
=-\theta^2_2(2\tau)\left(\theta^4_3(2\tau)+\theta^4_4(2\tau)\right)
\left(1385\theta^{12}_2(2\tau)+410\theta^4_2(2\tau)\theta^4_3(2\tau)\theta^4_4(2\tau)\right);
\end{align*}

\begin{align*}
&2^3\sum_{n=0}^{\infty}\sigma_{10,\chi}(4n+1)q^{4n+1}\\
=&\theta^2_2(2\tau)\left(\theta^4_3(2\tau)+\theta^4_4(2\tau)\right)
\left(50521\theta^{16}_2(2\tau)+38147\theta^8_2(2\tau)\theta^4_3(2\tau)\theta^4_4(2\tau)
+\theta^8_3(2\tau)\theta^8_4(2\tau)\right),\\
&2^3\sum_{n=0}^{\infty}\sigma_{10,\chi}(4n+3)q^{4n+3}\\
=&-\left(50521\theta^{22}_2(2\tau)+138677\theta^{14}_2(2\tau)\theta^4_3(2\tau)\theta^4_4(2\tau)
+7381\theta^6_2(2\tau)\theta^8_3(2\tau)\theta^8_4(2\tau)\right);
\end{align*}

\begin{align*}
&\theta^6_2(2\tau)
=-8\sum_{n=0}^{\infty}\sigma_{2,\chi}(4n+3)q^{4n+3},\\
&5\theta^{10}_2(2\tau)
=8\sum_{n=0}^{\infty}\sigma_{4,\chi}(4n+1)q^{4n+1}-8q\frac{(q^4;q^4)^{14}_\infty}{(q^8;q^8)^4_\infty},\\
&61\theta^{14}_2(2\tau)
=-8\sum_{n=0}^{\infty}\sigma_{6,\chi}(4n+3)q^{4n+3}-91\times 2^6q^3(q^4;q^4)^{10}_\infty(q^8;q^8)^{4}_\infty,\\
&1385\theta^{18}_2(2\tau)
=8\sum_{n=0}^{\infty}\sigma_{8,\chi}(4n+1)q^{4n+1}-q\frac{(q^4;q^4)^{30}_\infty}{(q^8;q^8)^{12}_\infty}\\
 &\qquad\qquad\qquad -763\times 2^9q^5(q^4;q^4)^{6}_\infty(q^8;q^8)^{12}_\infty,\\
& 50521\theta^{22}_2(2\tau)
=-8\sum_{n=0}^{\infty}\sigma_{10,\chi}(4n+3)q^{4n+3}\\
&\quad\quad\quad \quad\quad\quad-138677\times2^{14}q^7(q^4;q^4)^{2}_\infty(q^8;q^8)^{20}_\infty
 -7381\times2^6q^3\frac{(q^4;q^4)^{26}_\infty}{(q^8;q^8)^{4}_\infty}.
\end{align*}

Lastly, for $\theta_2^2(q)$, we recall  \cite[p. 511]{whittaker and watson}
\begin{align*}
cd u=\frac{2\pi}{Kk}\sum_{n=0}^{\infty}\frac{(-1)^nq^{n+\frac12}\cos(2n+1)z}{1-q^{2n+1}}
\end{align*}
or equivalently,
\begin{align*}
\theta_2(q)\theta_3(q)\frac{\theta_2(z|q)}{\theta_3(z|q)}=4\sum_{n=0}^{\infty}\frac{(-1)^nq^{n+\frac12}\cos(2n+1)z}{1-q^{2n+1}}.
\end{align*}
Set $z=0$ in the above equation, we have
\begin{align*}
\theta_2^2(q)=4\sum_{n=0}^{\infty}\frac{(-1)^nq^{n+\frac12}}{1-q^{2n+1}}.
\end{align*}



\end{document}